\documentclass[12pt]{amsart}
\usepackage{amssymb, xcolor}
\usepackage{amsmath}
\usepackage{graphicx}
\usepackage{framed}
\usepackage{mathrsfs}
\usepackage{url}
\usepackage{cite}
\usepackage{fullpage}
\usepackage[hidelinks]{hyperref}
\usepackage{stmaryrd}
\SetSymbolFont{stmry}{bold}{U}{stmry}{m}{n}

\usepackage{array}
\usepackage{enumitem}
\usepackage{mathscinet}
\usepackage{tikz}
\usetikzlibrary{arrows.meta,shapes.geometric}
\usepackage{placeins}

\definecolor{pointcountgreen}{RGB}{40,115,75}
\definecolor{weightblue}{RGB}{33,102,172}
\definecolor{hodgeorange}{RGB}{179,88,6}

\newtheorem{thm}{Theorem}[section]
\newtheorem{prop}[thm]{Proposition}
\newtheorem{lem}[thm]{Lemma}
\newtheorem{cor}[thm]{Corollary}

\definecolor{inputblue}{RGB}{240,246,251}
\definecolor{inputorange}{RGB}{255,247,229}
\newenvironment{inputbox}
 {%
  \MakeFramed{\advance\hsize-2\width\FrameRestore}\noindent}
 {\endMakeFramed}
 
\theoremstyle{definition}
\newtheorem{rem}[thm]{Remark}

\newtheorem{computerverification}[thm]{Computer Verification}

\numberwithin{equation}{section}

\newcommand{\Z}{\mathbb{Z}}
\newcommand{\F}{\mathbb{F}}

\newcommand{\R}{\mathbb{R}}

\newcommand{\M}{\mathcal{M}}

\newcommand{\incoeff}{\mathrel{\in_{\mathrm{coeff}}}}

\newcommand{\ZZ}{\mathbb{Z}}
\newcommand{\CC}{\mathbb{C}}
\newcommand{\QQ}{\mathbb{Q}}
\newcommand{\Q}{\mathbb{Q}}

\newcommand{\bbS}{\mathbb{S}}

\newcommand{\gr}{\mathrm{gr}}

\DeclareMathOperator{\vdim}{dim}

\newcommand{\Ecal}{\mathcal E}
\newcommand{\Hcal}{\mathcal H}
\newcommand{\Rcal}{\mathcal R}
\newcommand{\Fcal}{\mathcal F}

\newcommand{\Qcal}{\mathcal Q}
\newcommand{\Mcal}{\mathbf M}
\newcommand{\Sfac}{\mathsf S}
\newcommand{\Trunc}[1]{T_{\leq #1}}
\newcommand{\rise}[2]{\mathopen{}\left(#1\right)^{\overline{#2}}\mathclose{}}
\newcommand{\fall}[2]{\mathopen{}\left(#1\right)^{\underline{#2}}\mathclose{}}

\author{Sam Payne}
\email{sdpayne@umich.edu}
\author{Thomas Willwacher}
\email{thomas.willwacher@math.ethz.ch}
\thanks{
S.P. was supported in part by NSF grant DMS--2542134 and a Simons Fellowship.}

\title{Polynomial point counts for moduli spaces of curves with marked points}

\begin{document}

\begin{abstract}
    We prove that the number of curves of a fixed genus $g$ with $n$ marked points over finite fields is a polynomial function of the cardinality of the field if and only if $g=0$ or $3g+2n \leq 24$. This confirms a conjecture of Canning, Larson, and the authors.
\end{abstract}

\maketitle

\setcounter{tocdepth}{1}
\tableofcontents

\section{Introduction}
Let $N_{g,n}(q)$ be the number of geometric isomorphism classes of smooth projective curves of genus $g$ with $n$ marked points that are defined over a finite field $\F_q$ of order $q$. This agrees with the stacky point count $\# \M_{g,n}(\F_q)$ when $2g + n > 2$.

\begin{thm}\label{thm:main pointcounts}
 The function  $N_{g,n}(q)$ is a polynomial in $q$ if and only if $g=0$ or $3g+2n\leq 24$.
\end{thm}

\noindent This confirms a conjecture of Canning, Larson, and the authors  \cite[Conjecture~1.6]{CLPW}, who proved it for $n = 0$ and for $g + n < 150$. The case $g=0$ was also previously known; the remaining cases are new.
Figure~\ref{fig:pointcount-regions} illustrates the polynomial range and the division of the proof.

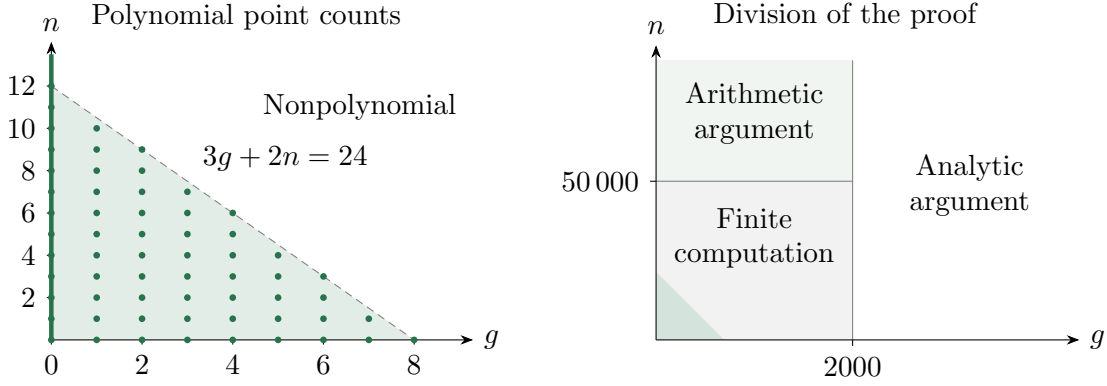
\begin{figure}[tbp]
\centering
\begin{tikzpicture}[font=\small,>=Stealth]
  \begin{scope}[x=.60cm,y=.28cm]
    \node at (4.3,15.3) {Polynomial point counts};
    \fill[pointcountgreen!13] (0,0)--(8,0)--(0,12)--cycle;
    \draw[gray,densely dashed] (0,12)--(8,0);
    \draw[->] (0,0)--(9.3,0) node[right] {$g$};
    \draw[->] (0,0)--(0,14.1) node[above] {$n$};
    \foreach \gg in {0,2,4,6,8} {
      \draw (\gg,0)--(\gg,-.22) node[below] {$\gg$};
    }
    \foreach \nn in {2,4,6,8,10,12} {
      \draw (0,\nn)--(-.10,\nn) node[left] {$\nn$};
    }
    \foreach \gg in {1,...,8} {
      \foreach \nn in {0,...,12} {
        \pgfmathtruncatemacro{\polyflag}{3*\gg+2*\nn}
        \ifnum\polyflag<25
          \fill[pointcountgreen] (\gg,\nn) circle[radius=1.25pt];
        \fi
      }
    }
    \draw[pointcountgreen,line width=1.7pt] (0,0)--(0,13.5);
    \foreach \nn in {0,...,12} {
      \fill[pointcountgreen] (0,\nn) circle[radius=1.25pt];
    }
    \node[anchor=west] at (3.1,8.6) {$3g+2n=24$};
    \node[align=center] at (6.8,11.0) {Nonpolynomial};
  \end{scope}
  \begin{scope}[shift={(8cm,0)},x=1cm,y=1cm]
    \node at (2.5,4.284) {Division of the proof};
    \fill[black!5] (0,0) rectangle (2.6,2.1);
    \fill[pointcountgreen!7] (0,2.1) rectangle (2.6,3.7);
    \fill[pointcountgreen!20] (0,0)--(.9,0)--(0,.9)--cycle;
    \draw[gray] (2.6,0)--(2.6,3.7);
    \draw[gray] (0,2.1)--(2.6,2.1);
    \draw[->] (0,0)--(5.6,0) node[right] {$g$};
    \draw[->] (0,0)--(0,3.9) node[above] {$n$};
    \draw (2.6,0)--(2.6,-.08) node[below] {$2000$};
    \draw (0,2.1)--(-.08,2.1) node[left] {$50\,000$};
    \node[align=center] at (1.3,3.0) {Arithmetic\\argument};
    \node[align=center] at (1.3,1.35) {Finite\\computation};
    \node[align=center] at (4.15,2.05) {Analytic\\argument};
  \end{scope}
\end{tikzpicture}
\caption{Polynomial point counts occur precisely on the green $g=0$ axis and at the lattice points satisfying $3g+2n\leq24$. The right panel shows the division of the proof schematically.}
\label{fig:pointcount-regions}
\end{figure}

We prove nonpolynomiality by detecting nonzero Euler characteristics in odd weight or in Hodge type $(p,q)$ with $p\neq q$. Katz's theorem implies that polynomial point counts force all such Euler characteristics to vanish. We recall the precise statement in Section~\ref{sec:pointcounts}.

We consider the Euler characteristics
\[
\chi_{11}(\M_{g,n}) := \sum_{j}(-1)^j \vdim_\QQ \gr_{11} H_c^j{\M_{g,n}},
\]
and
\[
\chi_{15,0}(\M_{g,n}) := \sum_{j}(-1)^j \vdim_\CC  \gr_{15,0} \big(H_c^j{\M_{g,n}} \otimes \CC \big).
\]
Here, $\gr_k$ denotes the $k$th graded piece of Deligne's weight filtration on cohomology with rational coefficients, and $\gr_{p,q}(H\otimes\CC)$ denotes the summand of Hodge type $(p,q)$ in the complexification of the rational Hodge structure $\gr_{p+q}H$.

\begin{thm}\label{thm:main nonvanishing}
Let $g\geq 1$ and $n\geq 0$ be such that $3g+2n\geq 25$ and $(g,n)\notin\{(12,0),(8,1)\}$. Then $\chi_{11}(\M_{g,n})\neq 0$ or $\chi_{15,0}(\M_{g,n})\neq 0$.
\end{thm}

\noindent The nonpolynomiality of $N_{12,0}$ and $N_{8,1}$ was proved in \cite{CLPW} by showing that the weight-$13$ Euler characteristic $\chi_{13}$ does not vanish in these two cases.

\medskip

The proof of Theorem~\ref{thm:main nonvanishing} has two parts, both based on the generating-function formulas for $\chi_{11}$ and $\chi_{15,0}$ recalled in Section~\ref{sec:generating-functions}. The bounded-genus argument in Section~\ref{sec:mainA} is arithmetic. For $g\leq 2000$, a finite expansion of the weight-$11$ generating function expresses the dependence on $n$ in terms of rising factorials. The behavior of rising factorials modulo primes reduces nonvanishing for large $n$ to finding primes in suitable intervals, which is proved using well-known prime-counting estimates. Finite exact computations cover the remaining small values of $n$. This shows that $\chi_{11}$ is nonzero throughout the required range, apart from $(12,0)$ and $(8,1)$, which were treated in \cite{CLPW} using $\chi_{13}$.

The high-genus argument in Section~\ref{sec:mainB} is analytic. For $g>2000$, we build on the argument for $n=0$ in \cite[Section~6]{CLPW}, which separates a leading contribution from a remainder and proves that the leading contribution dominates. The main difficulty is to obtain such domination uniformly in the number of marked points. For each parity of $g$, as $g$ grows with $n/g$ fixed, the normalized leading contribution, up to an overall sign, has a limit that vanishes at finitely many ratios, which we call \emph{critical slopes}. Near the corresponding lines $n=cg$, cancellation makes the leading contribution small. These troublesome regions extend to arbitrarily large $g$, with $n$ increasing proportionally: estimates for each fixed $n$ do not provide the required uniform control, and computations in any bounded range of genera leave these regions beyond that range untreated.

We overcome this difficulty by considering $\chi_{11}$ and $\chi_{15,0}$ together. For each parity of $g$, their sets of critical slopes are disjoint. Near a critical slope for one Euler characteristic, the limiting leading contribution for the other remains nonzero. For every $g>2000$, uniform estimates show that at least one leading contribution dominates its remainder.

Figure~\ref{fig:first-critical-slopes} shows the three smallest critical slopes for each parity, with different vertical scales in the two panels. The complete sets are displayed in Figure~\ref{fig:all-critical-slopes}, accompanying Remark~\ref{rem:critical-slopes} in Section~\ref{sec:mainB}.

\begin{figure}[tbp]
\centering
\begin{tikzpicture}[font=\small,>=Stealth]
  \begin{scope}[x=.85cm,y=.16cm]
    \node at (3,29) {Even genus};
    \fill[black!5] (0,0) rectangle (2,24);
    \draw[gray,densely dotted] (2,0)--(2,24);
    \draw[->] (0,0)--(6.35,0) node[right] {$g$};
    \draw[->] (0,0)--(0,25.5) node[above] {$n$};
    \foreach \xx/\lab in {0/0,2/2000,4/4000,6/6000} {
      \draw (\xx,0)--(\xx,-.45) node[below,font=\footnotesize] {$\lab$};
    }
    \foreach \yy/\lab in {5/5000,10/10000,15/15000,20/20000} {
      \draw (0,\yy)--(-.08,\yy) node[left,font=\footnotesize] {$\lab$};
    }
    \draw[weightblue,line width=1pt] (0,0)--(6,2.0105030345);
    \draw[hodgeorange,line width=1pt,dashed] (0,0)--(6,1.0430417351);
    \draw[hodgeorange,line width=1pt,dashed] (0,0)--(6,21.3503940769);
  \end{scope}
  \begin{scope}[shift={(8cm,0)},x=.85cm,y=.075cm]
    \node at (3,61.8667) {Odd genus};
    \fill[black!5] (0,0) rectangle (2,51.2);
    \draw[gray,densely dotted] (2,0)--(2,51.2);
    \draw[->] (0,0)--(6.35,0) node[right] {$g$};
    \draw[->] (0,0)--(0,54.4) node[above] {$n$};
    \foreach \xx/\lab in {0/0,2/2000,4/4000,6/6000} {
      \draw (\xx,0)--(\xx,-.96) node[below,font=\footnotesize] {$\lab$};
    }
    \foreach \yy/\lab in {10/10000,20/20000,30/30000,40/40000,50/50000} {
      \draw (0,\yy)--(-.08,\yy) node[left,font=\footnotesize] {$\lab$};
    }
    \draw[weightblue,line width=1pt] (0,0)--(6,14.3247750676);
    \draw[hodgeorange,line width=1pt,dashed] (0,0)--(6,7.7439064865);
    \draw[hodgeorange,line width=1pt,dashed] (0,0)--(6,50.4864346802);
  \end{scope}
  \draw[weightblue,line width=1pt] (4.4,-.85)--(5.1,-.85) node[right,black] {$\chi_{11}$};
  \draw[hodgeorange,line width=1pt,dashed] (7.2,-.85)--(7.9,-.85) node[right,black] {$\chi_{15,0}$};
\end{tikzpicture}
\caption{The three smallest critical slopes for even and odd genus.}
\label{fig:first-critical-slopes}
\end{figure}
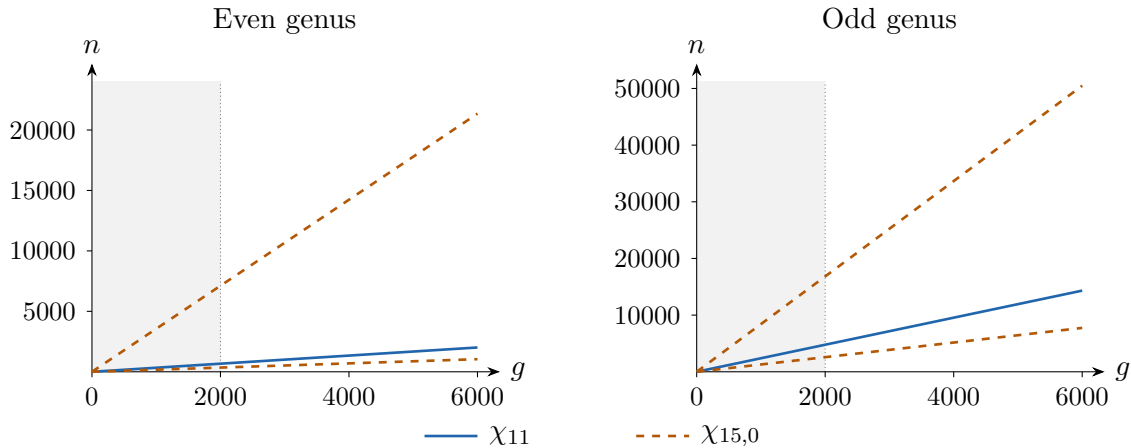

The use of both Euler characteristics in the analytic argument is facilitated by the special form of their generating functions. Up to normalization, the generating functions for $\chi_{11}$ and $\chi_{15,0}$ are obtained by applying the truncations $T_{\leq 10}$ and $T_{\leq 14}$, respectively, to the same series, a one-parameter deformation of the generating function for weight-$0$ Euler characteristics conjectured by Zagier and proved in \cite{CFGP}. The same expansion and coefficient estimates therefore apply to both, while the different truncations give different sets of critical slopes. In particular, the finite expansions of Section~\ref{sec:generating-functions}, on which the arithmetic argument rests, are deformed versions of the finiteness of the weight-$0$ generating function in each fixed genus.

Combining the new results of this paper with \cite[Theorem~1.5 and Remark~3.8]{CLPW} gives the following stronger form of Theorem~\ref{thm:main pointcounts}. See Section~\ref{sec:pointcounts}.

\begin{cor}\label{cor:tautological-tate-pointcounts}
Assume $2g-2+n>0$. Then the following are equivalent:
\begin{enumerate}[label=(\roman*)]
\item The $E_1$-page of the weight spectral sequence associated to $\M_{g,n}\subset\overline{\M}_{g,n}$ is generated by tautological classes.
\item The rational cohomology of $\M_{g,n}$ is of Tate type.
\item There is a polynomial $P \in \Z[T]$ such that $\#\M_{g,n}(\F_q)=P(q)$ for all finite fields $\F_q$.
\item Either $g=0$ or $3g+2n\leq24$.
\end{enumerate}
\end{cor}

The computational verifications used in our proofs are highlighted in boxes below. Code for each verification is available in the following GitHub repository:

\smallskip 

\url{https://github.com/wilthoma/mgn_pointcounts2_code}

\subsection*{Acknowledgments}
We are especially grateful to Sam Canning and Hannah Larson for the many discussions in the course of our joint work that led to the formulation of the conjecture proved here.

\subsection*{Declaration of AI usage}
ChatGPT 5.6 (OpenAI) was used in this work both for generating candidate proof ideas and strategies, for arranging the various estimates, and for writing code to set up the exact computations underlying the arithmetic proof for small genus and the analytic proof for large genus. The code was written in Python and C++, using NumPy \cite{numpy}, FLINT \cite{flint} and Arb \cite{Johansson2017arb}. %
The authors assume responsibility for all content.
\FloatBarrier
\section{Polynomial point counts and Hodge--Deligne polynomials} \label{sec:pointcounts}
The \emph{Hodge--Deligne polynomial} of a complex algebraic variety or Deligne--Mumford stack $X$ records the Euler characteristics of the Hodge types in its compactly supported cohomology:
\[
 E(X;x,y):=
 \sum_{p,q}\Big(\sum_j(-1)^j
 \dim_{\CC}\gr_{p,q}\bigl(H_c^j(X)\otimes\CC\bigr)\Big)x^py^q.
\]
Here cohomology is taken with rational coefficients. In particular, the coefficient of $t^k$ in $E(X;t,t)$ is the weight-$k$ Euler characteristic, whereas the coefficient of $x^py^q$ in $E(X;x,y)$ is the Hodge type $(p,q)$ Euler characteristic.

A theorem of Katz \cite[Theorem~6.1.2(3)]{Katz2008} implies that if $X$ is a separated scheme of finite type over $\ZZ$ and $P(T)\in\ZZ[T]$ satisfies $\#X(\F_q)=P(q)$ for every finite field $\F_q$, then
\[
 E(X_{\CC};x,y)=P(xy).
\]

For $2g+n\ge3$, let $M_{g,n}$ be the coarse moduli space of $\M_{g,n}$, defined over $\ZZ$. Thus, by definition, $N_{g,n}(q)=\#M_{g,n}(\F_q)$. We also have
\[
 \#M_{g,n}(\F_q)=\#\M_{g,n}(\F_q),
\]
by \cite[Proposition~1.3(iii)]{BergstromFaberPayne}. Over $\CC$, the proper coarse moduli map induces an isomorphism on compactly supported rational cohomology, compatible with mixed Hodge structures. Thus, if $N_{g,n}$ is a polynomial $P$, $P\in\ZZ[T]$ \cite[pp.~616--617]{Katz2008}. \mbox{Katz's theorem gives}
\begin{equation}\label{eq:polynomial-hodge-deligne}
 E(\M_{g,n};x,y)=P(xy).
\end{equation}

\begin{proof}[\normalfont\bfseries Proof that Theorem~\ref{thm:main pointcounts} follows from Theorem~\ref{thm:main nonvanishing}]
We have $N_{0,n}(q)=1$ for $n\le2$ and $N_{1,0}(q)=q$. Polynomiality in the remaining cases with $g=0$ or $3g+2n\le24$ follows from \cite[Theorem~1.5]{CLPW}.

It remains to show that $N_{g,n}$ is not polynomial for $g\ge1$ and $3g+2n\ge25$. By \eqref{eq:polynomial-hodge-deligne}, it suffices to show that $\M_{g,n}$ has a nonzero Euler characteristic in odd weight or in Hodge type $(p,q)$ with $p\neq q$. Theorem~\ref{thm:main nonvanishing} gives the required nonvanishing for $(g,n)\notin\{(12,0),(8,1)\}$. For the two exceptional pairs, $\chi_{13}(\M_{g,n})\neq0$ by \cite[Corollary~1.8]{CLPW}.
\end{proof}

A rational mixed Hodge structure $H$ is \emph{of Tate type} if $\gr_{2k+1}H=0$ and $\gr_{2k}H$ is a direct sum of copies of $\QQ(-k)$ for every $k$. Here $\QQ(-k)$ denotes the one-dimensional rational Hodge structure of type $(k,k)$. We say that the cohomology of a variety or Deligne--Mumford stack is of Tate type if each of its cohomology groups is of Tate type.

The $E_1$-page of the weight spectral sequence for $\M_{g,n}\subset\overline{\M}_{g,n}$ has a description in terms of stable graphs with vertex decorations in the cohomology of moduli spaces of stable curves, modulo the signed identifications induced by graph automorphisms; see \cite[\S3.1]{CLPW}. It is \emph{generated by tautological classes} if it is spanned by graphs whose vertex decorations lie in the tautological rings. The tautological rings form the smallest collection of subrings of $H^*(\overline{\M}_{g,n})$ containing the unit classes and closed under pullback and pushforward along gluing and forgetful maps.

\begin{proof}[\normalfont\bfseries Proof that Corollary~\ref{cor:tautological-tate-pointcounts} follows from Theorem~\ref{thm:main pointcounts}]
First, (ii) and (iv) are equivalent by \cite[Theorem~1.5]{CLPW}. If (iv) holds, then (i) follows from \cite[Remark~3.8]{CLPW} when $3g+2n\leq24$, and from \cite{Keel} when $g=0$. Since tautological classes are algebraic, (i) implies that every term on the $E_1$-page is of Tate type. The same is therefore true of the weight-graded compactly supported cohomology, and (ii) follows by Poincar\'e duality. Finally, Theorem~\ref{thm:main pointcounts} gives the equivalence of (iii) and (iv).
\end{proof}

\section{Generating functions for \texorpdfstring{$\chi_{11}$ and $\chi_{15,0}$}{chi11 and chi15,0}}\label{sec:generating-functions}

\subsection{Generating-function formulas}
We recall the generating functions for the $\bbS_n$-equi\-variant compactly supported Euler characteristics of $\M_{g,n}$ in weight $11$ and Hodge type $(15,0)$. Both are built from a product of functions $U_\ell$, defined below, which were introduced by Songhafouo Tsopm\'en\'e and Turchin \cite{TsopmeneTurchin} and which package the all-genus weight-$0$ generating function; see Remark~\ref{rem:CFGP}.

The Bernoulli polynomials $B_r(x)$ and Bernoulli numbers $B_r$ are defined by
\begin{equation}\label{eq:Bernoulli}
 \frac{te^{xt}}{e^t-1}=\sum_{r\ge0}B_r(x)\frac{t^r}{r!},
 \qquad B_r:=B_r(0),
 \qquad\mbox{and}\qquad B_1=-\frac12.
\end{equation}
Set
\[
 B(z):=\sum_{r\ge2}\frac{B_r}{r(r-1)z^{r-1}},
 \qquad E_\ell:=\frac1\ell\sum_{d\mid\ell}\mu(\ell/d)u^{-d},
 \qquad\mbox{and}\qquad \lambda_\ell:=\ell u^\ell(1-u^\ell),
\]
where $\mu$ is the M\"obius function. We define $U_\ell(X,u)=\exp(\log U_\ell(X,u))$ by
\begin{equation} \label{equ:Uelldef} 
\resizebox{.92\hsize}{!}{
$\begin{aligned}
    \log U_\ell(X,u) 
  &=
        X\left(\log(\lambda_\ell E_\ell)-1 \right)+(-E_\ell+X-\textstyle{\frac 1 2} )\log(1-\textstyle{\frac X{E_\ell}}) +
         B(-E_\ell+X)- B(-E_\ell). 
         \end{aligned}
          $}
\end{equation}
For later coefficient calculations, it is useful to record the completely algebraic expansion
\begin{align}
 \log U_\ell(X,u)
 ={}&X\log(\lambda_\ell E_\ell)
 +\frac12\sum_{k\ge1}\frac{X^kE_\ell^{-k}}k
 -\sum_{k\ge2}\frac{X^kE_\ell^{-(k-1)}}{k(k-1)}\notag\\
 &-\sum_{\substack{r\ge2\\r\text{ even}}}
 \frac{B_r}{r(r-1)}
 \sum_{k\ge1}\binom{r+k-2}{k}X^kE_\ell^{-(r+k-1)}.
 \label{eq:Uell-expanded}
\end{align}
Since
\begin{equation}\label{equ:uellEell}
 \ell u^\ell E_\ell(u)
 =1+\sum_{\substack{d\mid\ell\\d<\ell}}
   \mu(\ell/d)u^{\ell-d},
\end{equation}
we may expand $E_\ell^{-1}$ by the geometric series and $E_\ell^{-s}$ for $s\ge2$ by the generalized binomial theorem.  In particular, all coefficients are rational and computable by finite arithmetic.

For a polynomial or power series in $w$, let $T_{\leq r}$ denote the sum of the coefficients of degrees at most $r$:
\[
T_{\leq r}\Bigl(\sum_k a_k w^k\Bigr)
:=
\sum_{k=0}^r a_k.
\]
The weight-$11$ generating function is given by the following formula; see \cite{PayneWillwacher24}.
\begin{equation}
\label{equ:ec11_withn}
\resizebox{.92\hsize}{!}{$
    \frac12 \displaystyle{\sum_{\substack{g,n\geq 0 \\ 2g+n\geq 3}}} u^{g+n} \chi^{\bbS_n}(\gr_{11}^W H^*_c(\M_{g,n}))
      =
      -u\  T_{\leq 10}\bigg(
      \prod_{\ell\geq 1} 
      \frac { 
        U_\ell(\frac 1 \ell \sum_{d\mid \ell} \mu(\ell/d) (-p_d  +1-w^d ), u )
      }
      { 
        U_\ell(\frac 1 \ell \sum_{d\mid \ell} \mu(\ell/d) (-p_d), u )
      }-1\bigg)\, .
      $}
\end{equation}
Similarly, for Hodge type $(15,0)$ we have (see \cite{CLPW2})
\begin{equation}
\label{equ:ec15_withn}
\resizebox{.92\hsize}{!}{$
    \displaystyle{\sum_{\substack{g,n\geq 0 \\ 2g+n\geq 3}}} u^{g+n} \chi^{\bbS_n}\big(\gr_{15,0}(H^*_c(\M_{g,n})\otimes\CC)\big)
      =
      -u\  T_{\leq 14}\bigg(
      \prod_{\ell\geq 1} 
      \frac { 
        U_\ell(\frac 1 \ell \sum_{d\mid \ell} \mu(\ell/d) (-p_d  +1-w^d ), u )
      }
      { 
        U_\ell(\frac 1 \ell \sum_{d\mid \ell} \mu(\ell/d) (-p_d), u )
      }-1\bigg)\, .
      $}
\end{equation}
\begin{rem}\label{rem:CFGP}
The functions $U_\ell$ of \eqref{equ:Uelldef} were introduced by Songhafouo Tsopm\'en\'e and Turchin \cite[Section~2.2]{TsopmeneTurchin}; see also \cite[Proposition~4.1]{PayneWillwacherEuler}. The product $$\prod_{\ell\geq1}U_\ell\Big(\frac1\ell\sum_{d\mid\ell}\mu(\ell/d)\,p_d,u\Big)$$ is the $\bbS_n$-equivariant Euler characteristic of the graph complex computing weight-$0$ compactly supported cohomology of $\M_{g,n}$, with disconnected graphs allowed \cite[Corollary~4.3]{PayneWillwacherEuler}. For each fixed $g\geq 2$, the weight-$0$ generating function $z_g$ for connected graphs is a finite Laurent polynomial in the inhomogeneous power sums $P_d=1+p_d$; see \cite[Remark~1.6]{CFGP}. The numerators in \eqref{equ:ec11_withn} and \eqref{equ:ec15_withn} mark one vertex by replacing $p_d$ with $p_d-1+w^d$, and the truncations $T_{\leq 10}$ and $T_{\leq 14}$ select the part of that vertex carrying the weight-$11$, respectively Hodge type $(15,0)$, class.
\end{rem}

From the $\bbS_n$-equivariant Euler characteristics one may obtain the non-equivariant Euler characteristics by evaluating the symmetric group character encoded by the symmetric function on the conjugacy class of the identity permutation.
After setting $p_1=p$ and $p_d=0$ for $d\geq 2$, the non-equivariant Euler characteristic is recovered by multiplying the coefficient of $p^n$ by $n!$. This gives the following formulas.

\begin{cor}\label{cor:EC formula}
The non-equivariant Euler characteristics are given by
\begin{equation}
\label{equ:ec11_non}
\resizebox{.92\hsize}{!}{$
    \frac12 \displaystyle{\sum_{\substack{g,n\geq 0 \\ 2g+n\geq 3}}} 
     u^{g+n} \frac{p^n}{n!} \chi(\gr_{11}^W H^*_c(\M_{g,n}))
      =
      -u\  T_{\leq 10}\bigg(
      \prod_{\ell\geq 1} 
      \frac { 
        U_\ell(\frac 1 \ell \sum_{d\mid \ell} \mu(\ell/d) (1-w^d ) -\frac {\mu(\ell)} \ell p, u )
      }
      { 
        U_\ell(-\frac {\mu(\ell)} \ell p, u )
      }-1\bigg)\, ,
      $}
\end{equation}
and
\begin{equation}
\label{equ:ec15_non}
\resizebox{.92\hsize}{!}{$
    \displaystyle{\sum_{\substack{g,n\geq 0 \\ 2g+n\geq 3}}} 
    u^{g+n} \frac{p^n}{n!} \chi\big(\gr_{15,0}(H^*_c(\M_{g,n})\otimes\CC)\big)
      =
      -u\  T_{\leq 14}\bigg(
      \prod_{\ell\geq 1} 
      \frac { 
        U_\ell(\frac 1 \ell \sum_{d\mid \ell} \mu(\ell/d) (1-w^d ) -\frac {\mu(\ell)} \ell p, u )
      }
      { 
        U_\ell(-\frac {\mu(\ell)} \ell p, u )
      }-1\bigg)\, .
      $}
\end{equation}
\end{cor}
To unify \eqref{equ:ec11_non} and \eqref{equ:ec15_non}, we introduce the notation
\begin{align}\label{equ:cGamma def}
    c^{(\Gamma)}_{g,n}
    :=
    [u^{g+n}p^n] 
    \left(-u\  T_{\leq \Gamma}\bigg(
      \prod_{\ell\geq 1} 
      \frac { 
        U_\ell(\frac 1 \ell \sum_{d\mid \ell} \mu(\ell/d) (1-w^d ) -\frac {\mu(\ell)} \ell p, u )
      }
      { 
        U_\ell(-\frac {\mu(\ell)} \ell p, u )
      }-1\bigg) \right)\, ,
\end{align}
where $[u^{g+n}p^n](\cdots)$ denotes the coefficient of $u^{g+n}p^n$ in the expression $(\cdots)$. Set
\[
 a_{g,n}:=c^{(10)}_{g,n} \qquad\mbox{and}\qquad b_{g,n}:=c^{(14)}_{g,n}.
\]
Then, by Corollary~\ref{cor:EC formula},
\[
 \chi(\gr_{11}^W H^*_c(\M_{g,n}))=2n!a_{g,n}
 \qquad\mbox{and}\qquad
 \chi\big(\gr_{15,0}(H^*_c(\M_{g,n})\otimes\CC)\big)=n!b_{g,n}.
\]

By these identities, Theorem~\ref{thm:main nonvanishing} is equivalent to showing that $a_{g,n}$ and $b_{g,n}$ do not vanish simultaneously in the stated range. We use $c^{(\Gamma)}_{g,n}$ to treat the two Euler characteristics together.

Finally, we note that for each fixed $g$ and $n$ the (a priori) infinite products over $\ell$ above may be restricted to a finite range.
\begin{lem}\label{lem:ell range}
Let $g\geq 1$ and $n\geq 0$. Then for computing the coefficient of $u^{g+n}p^n$ in \eqref{equ:ec11_non}, \eqref{equ:ec15_non} or  \eqref{equ:cGamma def} the infinite product over $\ell$ may be restricted to the range $1\leq \ell \leq \max \{1,2g-2\}$.
\end{lem}
\begin{proof}
    Assign $u^ap^b$ auxiliary degree $a-b$. We will show that the factors have no terms of negative auxiliary degree, and that every nonconstant term in the $\ell$th factor has auxiliary degree at least $\lceil\ell/2\rceil$ for $\ell\geq 2$. The prefactor $-u$ then gives the stated bound.

    We consider the logarithm of the $\ell$th term, 
    \[
    \log U_\ell\bigg(\frac 1 \ell \sum_{d\mid \ell} \mu(\ell/d) (1-w^d ) -\frac {\mu(\ell)} \ell p,u\bigg)
    -
    \log U_\ell\left( -\frac {\mu(\ell)} \ell p,u\right)
    ,
    \]
    and expand each summand using \eqref{eq:Uell-expanded}. 
    The first (logarithmic) term of \eqref{eq:Uell-expanded} contributes a $p$-independent term to the difference, which is furthermore of $u$-order at least $\lceil \ell/2\rceil$ since
    \[
 \lambda_\ell E_\ell(u)
 =(1-u^\ell) \bigg( 1+\sum_{\substack{d\mid\ell\\d<\ell}}
   \mu(\ell/d)u^{\ell-d} \bigg),
\]
and every proper divisor $d$ is at most $\ell/2$.

So we are left with studying the contributions of the remaining terms of \eqref{eq:Uell-expanded}, which are all multiples of $X^kE_\ell^{-s}$ with $s\geq \max\{1,k-1\}$. From \eqref{equ:uellEell} we see that $E_\ell^{-s}$ is of $u$-order $\ell s$. 
The factor $X^k$ can contribute terms of $p$-degree $k$, but the top powers cancel due to the subtraction, and hence the maximal contributed $p$-degree is $k-1$.
Hence any resulting monomial $u^ap^b$ has auxiliary degree
\[
a-b\geq \ell s - (k-1) \geq \ell-1 \geq \lceil \ell/2\rceil,
\]
with the last inequality true for $\ell\geq 2$.
The argument also shows that the factor with $\ell=1$ only contributes monomials of non-negative auxiliary degree. Exponentiation and products add auxiliary degrees, so the lemma follows.
\end{proof}

\subsection{Separating the $\ell=1$ factor}
For later use, we now rewrite and reorganize \eqref{equ:cGamma def}. Set
\[
 z:=-up \qquad\mbox{and}\qquad \tau:=\frac{u}{1-z}.
\]

Define the auxiliary function 
\begin{equation}\label{eq:H-definition}
\Hcal(\tau, w) := \exp\left( -\sum_{r\geq 1} \frac{B_{r+1}(1-w)-B_{r+1}}{r(r+1)}\tau^r\right),
\end{equation}
and denote its $h$th Taylor coefficient in $\tau$ by $H_h(w)$, so that 
\[
\Hcal(\tau, w) = \sum_{h\geq 0} H_h(w) \tau^h.
\]
The Bernoulli polynomials \eqref{eq:Bernoulli} satisfy \(B_{r+1}(1)=B_{r+1}\) for \(r\ge1\), and hence
\[
 H_0(w)=1,
 \qquad
 \Hcal(\tau,0)=1,
 \qquad\mbox{and}\qquad
 H_h(0)=0\quad(h\ge1).
\]

Furthermore, define 
\begin{equation}\label{eq:Rcal-def}
 \Rcal(u,z,w):=(1-u)^{1-w}
 \prod_{\ell\ge2}
 \frac{U_\ell( \frac 1 \ell \sum_{d\mid \ell} \mu(\ell/d) (1-w^d ) +\frac{\mu(\ell)}{\ell}z/u,u)}{U_\ell(\frac{\mu(\ell)}{\ell} z/u,u)},
\end{equation}
and denote its Taylor coefficients by $R_{m,j}(w)\in\Q[w]$, so that
\begin{equation}\label{eq:R-expansion}
 \Rcal(u,z,w)=\sum_{m\ge0}\sum_{j=0}^{m}R_{m,j}(w)u^mz^j
 \qquad\mbox{and}\qquad R_{0,0}=1.
\end{equation}

Note that the factor with $\ell=1$ in \eqref{equ:cGamma def} is
\[
 (1-z)^{1-w}(1-u)^{1-w}\Hcal\!\left(\frac{u}{1-z},w\right),
\]
so the entire product factors as
\[
  (1-z)^{1-w}\Hcal\!\left(\frac{u}{1-z},w\right)\Rcal(u,z,w)
\]
and we obtain 
\begin{equation}
    \label{equ:series expanded}
\sum_{g,n} (-1)^{n+1} c^{(\Gamma)}_{g,n} u^{g-1}z^n = \Trunc{\Gamma} \biggl((1-z)^{1-w}\Hcal\!\left(\frac{u}{1-z},w\right)\Rcal(u,z,w) -1\biggr).
\end{equation}

To see the support condition $j\le m$ in \eqref{eq:R-expansion}, consider a term of degree $k$ in $X$ from \eqref{eq:Uell-expanded}. For us, only the difference of \eqref{eq:Uell-expanded} evaluated at $X=\frac 1\ell \sum_{d\mid\ell}\mu(\ell/d)(1-w^d)+\mu(\ell)z/(u\ell)$ and the evaluation at $X=\mu(\ell)z/(u\ell)$ is relevant. 
In this difference, a surviving monomial of $z$-degree $j$ has $j\le k-1$ and contributes $u^{-j}$.  The first sum in negative powers of $E_\ell$ contains $E_\ell^{-k}$, of $u$-valuation $\ell k$, while the second contains $E_\ell^{-(k-1)}$, of valuation $\ell(k-1)$.  Since $\ell\ge2$, the resulting $u$-degree is at least $j$ in both cases.  The Bernoulli terms have still larger valuation, and the linear logarithmic term has $z$-degree zero after subtraction.  The inequality $j\le m$ is preserved by products and exponentiation.

Let $[z^k](\cdots)$ again denote the coefficient of $z^k$ in the formal power series $(\cdots)$.
Then we record for later use:
\begin{prop}\label{prop:raw}
Extend \(c^{(\Gamma)}_{g,n}\) by zero when \(2g+n<3\), and let
\(\delta_{r,s}\) denote the Kronecker delta.  For \(g\ge1\) and \(n\ge0\),
\begin{align}\label{eq:raw-first}
 (-1)^{n+1}c^{(\Gamma)}_{g,n}
 ={}&\sum_{m=0}^{g-1}\sum_{j=0}^{\min(m,n)}
 \Trunc{\Gamma}\!\Bigl(
 R_{m,j}(w)H_{g-1-m}(w)
 [z^{n-j}](1-z)^{2-g+m-w}
 \Bigr) -\delta_{g,1}\delta_{n,0}.
\end{align}
\end{prop}
\begin{proof}
This follows by collecting the coefficient of $u^{g-1}z^n$ on the right-hand side of \eqref{equ:series expanded}.
\end{proof}

\section{Arithmetic proof of nonvanishing for \texorpdfstring{$g\leq 2000$}{g <= 2000}}
\label{sec:mainA}
In this section, we prove the following slight strengthening of Theorem~\ref{thm:main nonvanishing} for $g\leq 2000$.

\begin{thm}\label{thm:mainA}
Let $1\leq g\leq 2000$ and $n\geq 0$.
    If $3g+2n\geq 25$ and $(g,n)\neq (12,0), (8,1)$, then 
    $\chi_{11}(\M_{g,n})\neq 0$.
\end{thm}

\subsection{Finite expansion}
For \(g\ge1\), define
\begin{equation}\label{eq:Fg-def}
 \Fcal_g(z,w)
 :=[u^{g-1}]
 \biggl(
 (1-z)^{1-w}\Hcal\!\Bigl(\frac{u}{1-z},w\Bigr)
 \Rcal(u,z,w)
 \biggr).
\end{equation}
Then \eqref{eq:raw-first} becomes
\begin{equation}\label{eq:c-via-Fg}
 (-1)^{n+1}c^{(\Gamma)}_{g,n}
 =\Trunc{\Gamma}[z^n]\Fcal_g(z,w)-\delta_{g,1}\delta_{n,0}.
\end{equation}

\noindent A priori, the expressions $\Fcal_g(z,w)$ are infinite series in $z$. They become finite sums, however, when written in terms of $(1-z)^{e-w}$. This is a direct analogue, after forgetting the $\bbS_n$-action, of the fixed-genus weight-$0$ formulas conjectured by Zagier and proved in \cite[Theorem~1.1]{CFGP}. For each $g\geq 2$, those formulas express the generating function as a finite Laurent polynomial in the inhomogeneous power sum symmetric functions. See Remark~\ref{rem:finite-weight0}.

\begin{prop}\label{prop:finite-singular}
For each $g\geq 1$, we have the expansion
\begin{equation}\label{eq:Fg-singular}
 \Fcal_g(z,w)=\sum_{e=2-g}^{g}C_{g,e}(w)(1-z)^{e-w},
\end{equation}
where each $C_{g,e}(w)$ is a polynomial in $\Q[w]$. More explicitly,
\begin{equation}\label{eq:Cge-explicit}
 C_{g,e}(w)
 =\sum_{\substack{0\le m\le g-1\\0\le j\le m\\0\le\ell\le j\\
                   2-g+m+\ell=e}}
 (-1)^\ell\binom{j}{\ell}
 R_{m,j}(w)H_{g-1-m}(w).
\end{equation}
\end{prop}
\begin{proof}
Expanding \eqref{eq:Fg-def} by \eqref{eq:R-expansion} gives
\[
 \Fcal_g(z,w)
 =\sum_{m=0}^{g-1}\sum_{j=0}^{m}
 z^jR_{m,j}(w)H_{g-1-m}(w)(1-z)^{2-g+m-w}.
\]
Put $\xi=1-z$.  Since
\[
 z^j=(1-\xi)^j=\sum_{\ell=0}^j(-1)^\ell\binom{j}{\ell}\xi^\ell,
\]
we obtain terms with exponent $e=2-g+m+\ell$.  Grouping them by $e$ gives \eqref{eq:Fg-singular} and \eqref{eq:Cge-explicit}.  The inequalities $0\le\ell\le j\le m\le g-1$ imply $2-g\le e\le g$.
\end{proof}

\begin{rem}\label{rem:finite-weight0}
Under the non-equivariant specialization, $P_1=1-z$ and $P_d=1$ for $d\geq 2$. Thus, for each fixed $g\geq 2$, the weight-$0$ generating function is a Laurent polynomial in $1-z$; see \cite[Remark~1.6]{CFGP}. Proposition~\ref{prop:finite-singular} gives an analogous statement: $(1-z)^w\Fcal_g(z,w)$ is a Laurent polynomial in $1-z$, with coefficients in $\Q[w]$.
\end{rem}

Thus, for fixed $g$, the full dependence on $n$ is determined by finitely many polynomials $C_{g,e}$.  Also note that, after expanding $(1-z)^{-w}=\exp(w\log(1/(1-z)))$ and applying $\Trunc{\Gamma}$, the generating function in $n$ is a finite linear combination of terms
\[
 (1-z)^e\left(\log\frac1{1-z}\right)^k.
\]

\subsection{Rising factorials and coefficient extraction}
For \(m\ge0\), we use the rising and falling
factorials
\begin{equation}\label{eq:factorial-conventions}
 \rise{x}{m}:=\prod_{r=0}^{m-1}(x+r),
 \qquad
 \fall{x}{m}:=\prod_{r=0}^{m-1}(x-r),
 \qquad\mbox{and}\qquad
 \rise{x}{0}=\fall{x}{0}:=1.
\end{equation}

\begin{lem}\label{lem:rising-coefficient}
For every integer \(e\), every \(n\ge0\), and a formal variable \(X\),
\[
 n![z^n](1-z)^{e-X}=\rise{X-e}{n}.
\]
\end{lem}

\begin{proof}
The generalized binomial theorem gives
\[
 (1-z)^{-(X-e)}
 =\sum_{n\ge0}\frac{\rise{X-e}{n}}{n!}z^n.
\]
Since \((1-z)^{e-X}=(1-z)^{-(X-e)}\), comparison of coefficients proves the
identity.
\end{proof}

For \(r,n\ge0\), define
\begin{equation}\label{eq:D-def}
 D_r(g,n):=\Trunc{r}[z^n]\Fcal_g(z,w) =(-1)^{n+1}c^{(r)}_{g,n} +\delta_{g,1}\delta_{n,0}.
\end{equation}

Define the formal power series in \(X\)
\begin{equation}\label{eq:Pgn-def}
 \Qcal_{g,n}(X)
 :=\frac1{1-X}
 \sum_{e=2-g}^{g}C_{g,e}(X)\rise{X-e}{n}
 \in\Q\llbracket X\rrbracket,
\end{equation}
with $C_{g,e}(X)\in \Q[X]$ the polynomials of Proposition \ref{prop:finite-singular}. 

\begin{prop}\label{prop:P-encodes-D}
For every $g\ge1$, $n\ge0$, and $r\ge0$,
\begin{equation}\label{eq:Q-D}
 [X^r]\Qcal_{g,n}(X)=n!D_r(g,n).
\end{equation}
Consequently, away from $(g,n)=(1,0)$,
\begin{equation}\label{eq:Q-c}
 [X^\Gamma]\Qcal_{g,n}(X)
 =n!(-1)^{n+1}c^{(\Gamma)}_{g,n}.
\end{equation}
\end{prop}

\begin{proof}
By Proposition \ref{prop:finite-singular} and Lemma \ref{lem:rising-coefficient},
\[
 n![z^n]\Fcal_g(z,X)
 =\sum_{e=2-g}^{g}C_{g,e}(X)\rise{X-e}{n}.
\]
If $f(X)=\sum_{j\ge0}f_jX^j$, then
\[
 [X^r]\frac{f(X)}{1-X}=\sum_{j=0}^rf_j.
\]
Applying this identity to the preceding polynomial gives
\[
 [X^r]\Qcal_{g,n}(X)
 =n!\sum_{j=0}^r[w^j][z^n]\Fcal_g(z,w)
 =n!D_r(g,n).
\]
Equation \eqref{eq:Q-c} follows from \eqref{eq:D-def} and \eqref{eq:c-via-Fg}.
\end{proof}

For later use, let us also record here the following explicit formula for $D_1(g,n)$.

\begin{prop}\label{prop:D1-formula}
Let \(s>g\).  For \(g=1,2\),
\begin{equation}\label{eq:D1-g12}
 s!D_1(1,s)=-(s-2)!,
 \qquad
 s!D_1(2,s)=(s-2)!.
\end{equation}
For $g\geq 3$, put $h:=\lfloor(g-1)/2\rfloor$. Then
\begin{equation}\label{eq:D1-general}
 s!D_1(g,s)
 =(-1)^{g+1}\frac{B_{2h}}{2h}\rise{2h-1}{s}=(-1)^{g+1}
 \frac{B_{2h}}{2h(2h-2)!}(s+2h-2)!.
\end{equation}
\end{prop}
\begin{proof}
We use the Bernoulli polynomial identities \cite[Equations~(24.4.3) and~(24.4.34)]{NIST:DLMF}
\[
 B_{r+1}'(w)=(r+1)B_r(w)
 \qquad\mbox{and}\qquad
 B_r(1)=(-1)^rB_r.
\]
Since $\Hcal(\tau,0)=1$, logarithmic differentiation of \eqref{eq:H-definition} at $w=0$ gives
\[
 [w]\Hcal(\tau,w)
 =\sum_{r\geq 1 } \frac{B_{r+1}'(1)}{r(r+1)}\tau^r
 =
 \sum_{r\ge1}\frac{B_r(1)}r\tau^r
 =
 \sum_{r\ge1}\frac{(-1)^rB_r}r\tau^r
 .
\]
Moreover,
\[
 \Rcal(u,z,0)=1-u,
\]
because for $\ell\ge2$ one has
$\ell^{-1}\sum_{d\mid\ell}\mu(\ell/d)=0$.  The coefficient of $u^m$ in $[w]\Rcal$ is a polynomial in $z$ of degree at most $m$, by \eqref{eq:R-expansion}.

Put $L=\log(1/(1-z))$ and write
\[
 \Rcal(u,z,w)=(1-u)+wR_1(u,z)+O(w^2).
\]
Taking the coefficient of $w$ in \eqref{eq:Fg-def} gives
\begin{align}
 [w]\Fcal_g(z,w)
 =[u^{g-1}](1-z)\left(
 L(1-u)+(1-u)\sum_{r\ge1}\frac{(-1)^rB_r}r
 \left(\frac{u}{1-z}\right)^r+R_1(u,z)\right).
 \label{eq:Fg-linear}
\end{align}
The contribution of $R_1$ has $z$-degree at most $g$ and hence vanishes in degree $s>g$.  For $g\ge3$, the first term has no coefficient of $u^{g-1}$.  The remaining coefficient of $u^{g-1}$ is
\[
 \frac{(-1)^{g-1}B_{g-1}}{g-1}(1-z)^{2-g}
 -\frac{(-1)^{g} B_{g-2}}{g-2}(1-z)^{3-g}.
\]
All odd Bernoulli numbers above $B_1$ vanish.  Whether $g$ is odd or even, the coefficient of $z^s$ is therefore
\[
 (-1)^{g+1}\frac{B_{2h}}{2h}
 [z^s](1-z)^{-(2h-1)}.
\]
Lemma \ref{lem:rising-coefficient} gives \eqref{eq:D1-general}.

For $g=1$, the coefficient of $u^0$ in \eqref{eq:Fg-linear} is $(1-z)L$.  Hence, for $s>1$,
\[
 D_1(1,s)=[z^s](1-z)L=-\frac1{s(s-1)},
\]
which gives the first formula in \eqref{eq:D1-g12}.  For $g=2$, the coefficient of $u^1$ is $-(1-z)L$ plus a polynomial of degree at most two in $z$.  Thus, for $s>2$,
\[
 D_1(2,s)=-[z^s](1-z)L=\frac1{s(s-1)},
\]
and the second formula follows.    
\end{proof}

\subsection{Modular reduction via rising factorials}
The rising factorials have the following elementary property when considered modulo a prime number $p$.
\begin{lem}\label{lem:prime-period}
Let \(p\) be prime.  Then for every integer \(e\) and every \(n\ge0\) we have the following equality in $\F_p[X]$,
\begin{equation}\label{eq:prime-period-rising}
 \rise{X-e}{n+p}
 \equiv (X^p-X)\rise{X-e}{n}\pmod p.
\end{equation}
\end{lem}
\begin{proof}
Factor the rising factorial as
\[
 \rise{X-e}{n+p}
 =\rise{X-e}{n}\prod_{j=0}^{p-1}(X-e+n+j).
\]
The residues \(-e+n+j\), \(0\le j<p\), form a complete residue system
modulo \(p\).  Hence, writing \(\mathbb F_p\) for the field with \(p\)
elements,
\[
 \prod_{j=0}^{p-1}(X-e+n+j)
 \equiv\prod_{a\in\mathbb F_p}(X-a)=X^p-X\pmod p.
\]
This proves \eqref{eq:prime-period-rising}.
\end{proof}

We want to apply the lemma to the series $\Qcal_{g,n}(X)$ of \eqref{eq:Pgn-def}.
Since the coefficients of this series are a priori rationals and not integers, we have to be slightly careful. We say that the prime $p$ is \emph{admissible up to degree $W$} for $\Qcal_{g,n}(X)$ if all denominators appearing in the evaluation of $[X^k]\Qcal_{g,n}(X)$ via the formulas above have nonzero (reduced) denominators modulo $p$, for $k\leq W$.
In this case we may consider all such coefficients as elements of $\F_p$, consider the series $\Qcal_{g,n}(X)+O(X^{W+1})$ as an element of $\F_p\llbracket X \rrbracket / \F_p\llbracket X \rrbracket_{\geq W+1}$, and evaluate it using finite field arithmetic.
Also note that for our application we care about the series coefficients up to and including degree 14, but never beyond.

\begin{lem}\label{lem:large-prime-admissible}
Let $g\geq 1$ and $n\geq 0$. Then every prime $p>2g$
is admissible for $\Qcal_{g,n}$ up to any degree.
\end{lem}
\begin{proof}
By \eqref{eq:Pgn-def} and \eqref{eq:Cge-explicit} it suffices to control the denominators appearing in $R_{m,j}(w)$ and $H_h(w)$ for $m,h\leq g-1$. The case $g=1$ is immediate since $C_{1,1}=1$, so we assume $g\geq 2$.

By Lemma \ref{lem:ell range} only factors in the infinite product with $\ell\leq 2g-2$ can contribute to the coefficient of $u^{g-1}$ in $\Rcal$.
Moreover, look at the proof of Lemma \ref{lem:ell range}, but with the substitution $p=-z/u$ so that the auxiliary degree in the proof becomes the $u$-degree. The proof then shows that a term of degree $k$ in $X$ from \eqref{eq:Uell-expanded} has $u$-degree at least $k-1$, and a term involving a Bernoulli number $B_r$ has $u$-degree at least $\ell r$. Thus only $k\leq g$ and $r< g$ may occur.
The explicit divisors in these terms have the form $\ell,2,k,k-1,r,r-1$. By the von Staudt-Clausen-Theorem \cite[§24.10(i)]{NIST:DLMF} every prime $p$ dividing the denominator of $B_r$ ($r\geq 2$) is at most $r+1$. The expansion of the inverse of $\ell u^\ell E_\ell$ contains only integer coefficients, and hence the negative powers of $E_\ell$ appearing in \eqref{eq:Uell-expanded} have integer coefficients as well. Finally, $\log \Rcal$ has positive $u$-degree, so that the exponentiation can at worst produce denominators in the coefficient of $u^{g-1}$ that divide $(g-1)!$.
Hence the factor $R_{m,j}(w)$ can only contribute denominators with prime factors $\leq 2g-2$.

The definition of $H_h$ involves only Bernoulli polynomials of index $\leq h+1\leq g$. Their coefficients are integer multiples of Bernoulli numbers of no larger index, and hence at worst contribute primes $\leq g+1$ to the denominators.
Exponentiation and the division by $r$, $r+1$ in the definition of $H_h$ (see \eqref{eq:H-definition}) do not change that. It follows that any prime appearing in a relevant denominator is at most $\max\{2g-2,g+1\}$, which is less than $p$.
\end{proof}

From Lemma \ref{lem:prime-period} above we then immediately obtain:
\begin{cor}
\label{prop:P-translation}
Let \(p\) be admissible for $\Qcal_{g,n}$ up to degree $W$.  Then in $\F_p\llbracket X \rrbracket / \F_p\llbracket X \rrbracket_{\geq W+1}$ we have that
\[
 \Qcal_{g,n+p}(X)
 \equiv(X^p-X)\Qcal_{g,n}(X)\pmod p.
\]
More generally, if \(n=qp+s\) with \(q\ge0\) and \(0\le s<p\), then
\[
 \Qcal_{g,n}(X)
 \equiv(X^p-X)^q\Qcal_{g,s}(X)\pmod p.
\]
\end{cor}

Taking the coefficient of $X^\Gamma$ we hence obtain the following nonvanishing criterion for the numbers $c^{(\Gamma)}_{g,n}$.
\begin{cor}
\label{cor:modular-reduction}
Let \(\Gamma\in\{10,14\}\), let \(p>\Gamma\) be a prime that is admissible for $\Qcal_{g,n}$ up to degree $\Gamma$, and write
\[
 n=qp+s,
 \qquad 0\le s<p,
 \qquad 0\le q\le\Gamma-1.
\]
Then
\[
 n!D_\Gamma(g,n)
 \equiv(-1)^q s!D_{\Gamma-q}(g,s)\pmod p.
\]
For every pair \((g,n)\ne(1,0)\), this is equivalently
\[
 n!(-1)^{n+1}c^{(\Gamma)}_{g,n}
 \equiv(-1)^q s!D_{\Gamma-q}(g,s)\pmod p.
\]
In particular, if
\[
 D_{\Gamma-q}(g,s)\not\equiv0\pmod p,
\]
then \(c^{(\Gamma)}_{g,n}\ne0\), provided \((g,n)\ne(1,0)\).
\end{cor}
\begin{proof}
We have by \eqref{eq:Q-D}
\[
n!D_\Gamma(g,n)
=
[X^\Gamma]\Qcal_{g,n}(X)
\equiv [X^\Gamma] ( (X^p-X)^q\Qcal_{g,s}(X)) \pmod p.
\]
Since $p>\Gamma$ there is no contribution from the summand $X^p$ and the remaining statements easily follow.
\end{proof}

\subsection{Nonvanishing for large $n$}
Note that by Proposition \ref{prop:D1-formula} we see that $D_1(g,s)$ is nonzero as long as $s>g$. This holds also modulo a prime $p$, as long as $p$ is large enough and $B_{2h}\neq 0 \mod p$. Hence, if that holds and if in Corollary \ref{cor:modular-reduction} we can choose $p$ so that $q=\Gamma-1$, then we can conclude that \(c^{(\Gamma)}_{g,n}\ne0\) as desired.

To make this argument more precise, put for $g\ge3$
\begin{equation}\label{eq:Ag-def}
 h:=\left\lfloor\frac{g-1}{2}\right\rfloor,
 \qquad
 d_g:=2h-2,
 \qquad\mbox{and}\qquad
 (-1)^{g+1}\frac{B_{2h}}{2h}=\frac{A_g}{C_g}
\end{equation}
in lowest terms, with $C_g>0$.  For $g=1,2$, set $A_g=C_g=1$ and $d_g=0$. Thus $0\le d_g\le1996$ for $1\le g\le2000$. Then we have the following nonvanishing criterion.

\begin{prop}\label{prop:tail-prime-criterion}
Let $1\le g\le2000$, $\Gamma\in\{10,14\}$, and $n\ge0$.  Suppose there is a prime $p$ such that
\begin{equation}\label{eq:prime-window}
 \frac{n+d_g}{\Gamma}<p<\frac{n-g}{\Gamma-1},
 \qquad
 p>2\cdot2000+\Gamma+2,
 \qquad
 p\nmid A_g.
\end{equation}
Then $c^{(\Gamma)}_{g,n}\ne0$.
\end{prop}

\begin{proof}
Set $s=n-(\Gamma-1)p$.  The upper bound on $p$ gives $s>g$, while the lower bound gives $p>s+d_g$.  Hence $0\le s<p$ and
$n=(\Gamma-1)p+s$.  The large-prime condition makes $p$ admissible by Lemma \ref{lem:large-prime-admissible}.  Corollary \ref{cor:modular-reduction} with $q=\Gamma-1$ gives
\[
 n!(-1)^{n+1}c^{(\Gamma)}_{g,n}
 \equiv(-1)^{\Gamma-1}s!D_1(g,s)\pmod p.
\]
For $g=1,2$, the right-hand side is nonzero by \eqref{eq:D1-g12}, since $s<p$.  For $g\ge3$, admissibility implies that $C_g$ is invertible modulo $p$, and $p\nmid A_g$ makes $A_g/C_g$ nonzero.  Every factor in $\rise{2h-1}{s}$ is at most $s+d_g<p$.  Formula \eqref{eq:D1-general} is therefore nonzero modulo $p$.
\end{proof}

The above reduces the nonvanishing of $c^{(\Gamma)}_{g,n}$ to an existence problem for suitable primes. We will show that this problem is solvable when $n$ is sufficiently large. We begin with an elementary auxiliary lemma.

\begin{lem}\label{lem:large-prime-divisors}
Let $A\ne0$ be an integer and $L>1$.  The number of distinct prime divisors $p\mid A$ with $p>L$ is at most
\[
 \left\lfloor\frac{\log|A|}{\log L}\right\rfloor.
\]
\end{lem}

\begin{proof}
If $p_1,\dots,p_k>L$ are distinct prime divisors of $A$, then
$L^k<p_1\cdots p_k\le|A|$.  Taking logarithms proves the claim.
\end{proof}

Furthermore, we will use the explicit prime-counting estimates of Dusart \cite[Corollary 5.3]{Dusart}, with $\pi(x)$ the number of primes up to $x$:
\begin{equation}\label{eq:Dusart}
 \frac{x}{\log x-1}\le\pi(x)\quad(x\ge5393),
 \qquad
 \pi(x)\le\frac{x}{\log x-1.112}\quad(x\ge 4). %
\end{equation}

In the following we restrict to $\Gamma=10$, in order to show Theorem \ref{thm:mainA}.
Define the genus-independent prime window
\begin{equation}\label{eq:universal-prime-window}
 L(n):=\frac{n+1996}{10}
 \qquad\mbox{and}\qquad
 U(n):=\frac{n-2000}{9}-1.
\end{equation}
Every prime satisfying $L(n)<p\le U(n)$ satisfies the two strict inequalities in \eqref{eq:prime-window} (for $\Gamma=10$) for every $1\le g\le2000$.

We will need the following finite numerical data. 
\begin{inputbox}
\begin{computerverification}[\texttt{ag\_bounds.py}]\label{input:tail-endpoints}
Exact arithmetic verifies, for $A_g$ as in \eqref{eq:Ag-def},
\begin{align}
 |A_g|&<200000^{782} &&(1\le g\le2000),\label{eq:Ag-bound-10}
\end{align}
\end{computerverification}
\end{inputbox}

\begin{prop}\label{prop:infinite-prime-windows}
For every $1\le g\le2000$ and $n\ge2000000$, there is a prime satisfying the conditions of Proposition \ref{prop:tail-prime-criterion} for $\Gamma=10$.
In particular, $a_{g,n}\ne0$ for each such $(g,n)$.
\end{prop}

\begin{proof}
With $L(n)$ and $U(n)$ as in \eqref{eq:universal-prime-window}, consider
\begin{equation}\label{equ:ULdiff}
 F(n) := \frac{U(n)}{\log U(n)-1}
 -\frac{L(n)}{\log L(n)-1.112}.
\end{equation}
We first show that this function is increasing in $n\geq 2000000$. We compute
\[
 \frac{d}{dy}\frac{y}{\log y-a}
 =\frac{\log y-a-1}{(\log y-a)^2},
\]
and note that the last expression decreases as a function of $\log y$ once $\log y>a+2$.
We will also use
\[
 0<\log U(n)-\log L(n)
 <\log\frac{10}{9}<\frac1{9}.
\]
Write $\log L(n)=11+u$ with $u>0$, using that $\log L(n)\geq \log L(2000000)>11$. Thus $\log U(n)<\log L(n)+\frac 19 = 11+\frac 19 +u$. Then we bound the derivative of \eqref{equ:ULdiff} from below as
\begin{align*}
F'(n) &= \frac 19 \frac{\log U(n)-2}{(\log U(n)-1)^2}
-\frac{1}{10}
\frac{\log L(n)-2.112}{(\log L(n)-1.112)^2}
\\&>
\frac 19 \frac{11+\frac 19+u-2}{(11+\frac 1 9+u-1)^2}
-\frac{1}{10}
\frac{11+u-2.112}{(11+u-1.112)^2}
\\
&=
\frac{140625u^3+3779875u^2+34004765u+102686845}
{10(9u+91)^2(125u +1236)^2}
 >0.
\end{align*}

By explicit evaluation we have that 
\[
F(2000000)>1583,
\]
and hence, since $F$ is increasing, $F(n)>1583$ for all $n\geq 2000000$. In this range both of Dusart's estimates in \eqref{eq:Dusart} apply and show that the interval
$L(n)<p\le U(n)$ contains more than $1583$ primes, hence at least $1584$. Every such prime is larger than $200000$.  By the bound \eqref{eq:Ag-bound-10} of Computer Verification~\ref{input:tail-endpoints} and Lemma \ref{lem:large-prime-divisors}, at most $781$ of them divide $A_g$.  Thus at least one satisfies all conditions of Proposition \ref{prop:tail-prime-criterion}, so that $a_{g,n}\neq 0$.
\end{proof}

The proposition reduces Theorem \ref{thm:mainA} to a finite computation. The estimates in its proof are coarse, however, and the region of the $(g,n)$-table that remains to be covered is still large. We extend the range covered by a computer search as follows.
For a fixed prime $p$, the two strict inequalities in \eqref{eq:prime-window} are equivalent, for integral $n$, to
\[
 (\Gamma-1)p+g+1\le n\le\Gamma p-d_g-1.
\]
So we merely search through suitable primes to cover a larger interval of $n$'s.

\begin{inputbox}
\begin{computerverification}[\texttt{prime\_cover.py}]
\label{input:tail-cover}
For every $1\le g\le2000$, there is a finite list of primes
$\mathcal P_{g}$ such that the following holds:
\begin{enumerate}[label=(\roman*)]
 \item every prime in $\mathcal P_g$ exceeds $2\cdot2000+12$;
 \item every $p\in\mathcal P_{g}$ satisfies $p\nmid A_g$;
 \item the intervals
 \[
 [9p+g+1,\,10p-d_g-1],
 \qquad p\in\mathcal P_{g},
 \]
 cover every integer $50000\le n<2000000$;
\end{enumerate}
\end{computerverification}
\end{inputbox}
\noindent Combining Computer Verification~\ref{input:tail-cover} with Propositions~\ref{prop:tail-prime-criterion} and~\ref{prop:infinite-prime-windows}, we have:
\begin{cor}\label{cor:tail}
For \(1\le g\le2000\) and $n\ge50000$ we have $a_{g,n}\ne0$.
\end{cor}

\subsection{Completion of the arithmetic proof}
Corollary \ref{cor:tail} reduces the proof of Theorem \ref{thm:mainA} to a manageable finite computation. Namely, for each $g\leq 2000$ we compute $a_{g,n}$ for $n<50000$ modulo a suitable prime, and show that these numbers are nonzero in the range stated in Theorem \ref{thm:mainA}.
The program is available in the \href{https://github.com/wilthoma/mgn_pointcounts2_code}{GitHub repository} linked in the introduction.

Recall from \eqref{eq:D-def} that $D_{10}(g,n)=(-1)^{n+1}\chi_{11}(\M_{g,n})/(2n!)$ in the range below. Since $2n!$ is invertible modulo $2013265921$ for $n<50000$, the following verification proves the required nonvanishing of $\chi_{11}(\M_{g,n})$.

\begin{inputbox}
\begin{computerverification}[\texttt{finite\_cover.py}]
\label{input:finite-modular}
The program verifies that for every $1\le g\le2000$ and $0\le n<50000$ such that $3g+2n\geq 25$ and $(g,n)\notin \{(8,1),(12,0)\}$
\[
D_{10}(g,n) \not\equiv 0 \mod 2013265921.
\]
\end{computerverification}
\end{inputbox}

Computer Verification~\ref{input:finite-modular} and Corollary~\ref{cor:tail} together prove Theorem \ref{thm:mainA}.\hfill \qed

\section{Analytic proof of nonvanishing for \texorpdfstring{$g>2000$}{g > 2000}}
\label{sec:mainB}

We next turn to the analytic argument for high genus. The estimates below hold for $g\geq 2000$; together with Section~\ref{sec:mainA}, they show Theorem \ref{thm:main nonvanishing} by treating the remaining range $g>2000$.
This part is shown by a quite different argument, following roughly the strategy of \cite[Section 6]{CLPW}. 
Concretely, we split $c_{g,n}^{(\Gamma)}$ into a leading contribution and an error term. We may then conclude nonvanishing of $c_{g,n}^{(\Gamma)}$ from analytic estimates on the leading and error terms. Here it is essential to consider both $\Gamma=10$ and $\Gamma=14$, as we can only show that they do not vanish together, but not that they do not vanish individually.

\subsection{Leading terms, error bounds, and nonvanishing}
\label{sec:dominance}
Define
\begin{equation}\label{eq:Sfac}
 \Sfac_g:=\frac{(g-2)!}{(2\pi)^g}
 \qquad\mbox{and}\qquad
 \Sfac_{g,n}:=\binom{g+n-3}{n}\Sfac_g.
\end{equation}
With the rising factorials of \eqref{eq:factorial-conventions}, let
\begin{equation}\label{eq:NP}
 P_{g,n}(w):=\frac{\rise{g-2+w}{n}}{\rise{g-2}{n}},
\end{equation}
so that for $g\geq 3$ we have 
\begin{equation}\label{eq:binomial-P}
 [z^n](1-z)^{2-g-w}=\binom{g+n-3}{n} P_{g,n}(w).
\end{equation}
Furthermore, let for $h\ge 1$
\begin{equation}\label{eq:Hhat}
 \widehat H_h(w):=\frac{H_h(w)}{\Sfac_{h+1}},
\end{equation}
and for \(m\le g-3\), define
\begin{equation}\label{eq:omega}
 \omega_{m,j}(g,n):=(2\pi)^m\frac{g-m-2}{g-2}
 \frac{\fall{n}{j}}{\fall{g+n-3}{m+j}}.
\end{equation}
Then using the abbreviations \eqref{eq:Sfac}, \eqref{eq:NP}, \eqref{eq:Hhat}, and \eqref{eq:omega}, together with \eqref{eq:binomial-P}, in \eqref{eq:raw-first} we obtain
\begin{equation}\label{eq:normalized-full}
 \frac{(-1)^{n+1}c^{(\Gamma)}_{g,n}}{\Sfac_{g,n}}
 =\sum_{m=0}^{g-3}\sum_{j=0}^{\min(m,n)}
 \omega_{m,j}\Trunc{\Gamma}
 \bigl(R_{m,j}P_{g-m,n-j}\widehat H_{g-1-m}\bigr)
 +\partial_\Gamma(g,n),
\end{equation}
where \(\partial_\Gamma\) is the contribution of the terms with
\(m=g-2,g-1\) in the sum over $m$. We separate these terms only because the normalization factors used above become singular and zero, respectively; the original terms themselves are well-defined.

For $f(w)=\sum f_kw^k$, write
\begin{equation}\label{eq:norm}
 \|f\|_{1,\le14}:=\sum_{k=0}^{14}|f_k|
 \qquad\mbox{and}\qquad
 r_m:=\sum_{j=0}^m\|R_{m,j}\|_{1,\le14}.
\end{equation}
The polynomial $P_{g,n}$ has nonnegative coefficients, so
\begin{equation}\label{eq:Pnorm}
 \|P_{g,n}\|_{1,\le14}
 \le P_{g,n}(1)=\frac{g+n-2}{g-2}.
\end{equation}

\begin{lem}\label{lem:term-bound}
For $1\le m\le g-3$, the sum of the absolute values of all terms in \eqref{eq:normalized-full} with this fixed $m$ satisfies
\[
  \sum_{j=0}^{\min(m,n)} \left|
 \omega_{m,j}\Trunc{\Gamma}
 \bigl(R_{m,j}P_{g-m,n-j}\widehat H_{g-1-m}\bigr)
 \right|\leq 
 \frac{\|\widehat H_{g-1-m}\|_{1,\le14}r_m(2\pi)^m}
 {(g-2)\fall{g+n-3}{m-1}}.
\]
The absolute value of the sum of the final two terms with $m=g-2,g-1$ satisfies
\[
 |\partial_\Gamma(g,n)| \leq \frac{r_{g-2}+2r_{g-1}}{\Sfac_{g,n}}.
\]
\end{lem}

\begin{proof}
Multiply \eqref{eq:omega} by \eqref{eq:Pnorm}.  The factor $g-m-2$ cancels, as does the last factor of $\fall{g+n-3}{m+j}$.  Then use
\[
 \fall{g+n-3}{m+j-1}
 =\fall{g+n-3}{m-1}\fall{g+n-m-2}{j},
 \qquad
 \frac{\fall{n}{j}}{\fall{g+n-m-2}{j}}\le1,
\]
and sum over $j$. For the final two terms, we note that $H_0=1$ and $H_1=(w-w^2)/2$, so that $\|H_1\|_{1,\le14}=1=\|H_0\|_{1,\le14}$.
Also note that the sum of the absolute coefficients of $[z^N](1-z)^{-w}$ is one (evaluate at $w=1$), while the corresponding sum for $[z^N](1-z)^{1-w}$ is at most two. 
This yields 
\begin{align*}
\Sfac_{g,n}|\partial_\Gamma(g,n)|&\leq 
\sum_{j=0}^{\min(g-2,n)}
\left|
 \Trunc{\Gamma}
 \Bigl(R_{g-2,j}H_{1} [z^{n-j}](1-z)^{-w}\Bigr)
 \right|
 \\&\quad\quad\quad\quad\quad\quad+
 \sum_{j=0}^{\min(g-1,n)}
\left|
 \Trunc{\Gamma}
 \Bigl(R_{g-1,j} H_{0}[z^{n-j}](1-z)^{1-w} \Bigr)
 \right|
 \\
 &\leq 
 \sum_{j=0}^{\min(g-2,n)}
 \|R_{g-2,j}\|_{1,\leq 14}
 + 2
 \sum_{j=0}^{\min(g-1,n)}
\|R_{g-1,j}
 \|_{1,\leq 14} \leq r_{g-2}+2r_{g-1}. \qedhere
\end{align*}
\end{proof}
Expanding $\Rcal$ through $u^2$ gives
\begin{equation}\label{eq:R-low}
 R_{0,0}=1,
 \qquad
 R_{1,0}=-1+\frac12w+\frac12w^2,
 \qquad
 R_{1,1}=\frac12w-\frac12w^2,
 \qquad
 r_2=\frac{55}{12}.
\end{equation}
We define the leading term by keeping precisely the terms with $m=0,1$ in \eqref{eq:normalized-full}:
\begin{align}
 M_\Gamma(g,n):={}&
 \Trunc{\Gamma}\bigl(P_{g,n}\widehat H_{g-1}\bigr)\notag\\
 &+\omega_{1,0}\Trunc{\Gamma}\bigl(
 R_{1,0}P_{g-1,n}\widehat H_{g-2}\bigr)\notag\\
 &+\mathbf 1_{n\ge1}\,\omega_{1,1}\Trunc{\Gamma}\bigl(
 R_{1,1}P_{g-1,n-1}\widehat H_{g-2}\bigr).
 \label{eq:M-leading}
\end{align}
The error term is the scalar majorant of the remaining summands of \eqref{eq:normalized-full}
\begin{equation}\label{eq:E-error}
 \Ecal(g,n):=
 \sum_{m=2}^{g-3}
 \frac{\|\widehat H_{g-1-m}\|_{1,\le14}r_m(2\pi)^m}
 {(g-2)\fall{g+n-3}{m-1}}
 +\frac{r_{g-2}+2r_{g-1}}{\Sfac_{g,n}} \geq 0.
\end{equation}

We will derive the $g\geq 2000$-part of Theorem \ref{thm:main nonvanishing} from the following nonvanishing criterion.
\begin{prop}\label{prop:criterion}
For \(g\ge4\), \(n\ge0\), and \(\Gamma\in\{10,14\}\),
\[
 \left|
 \frac{(-1)^{n+1}c^{(\Gamma)}_{g,n}}{\Sfac_{g,n}}
 -M_\Gamma(g,n)\right|
 \le\Ecal(g,n).
\]
Consequently, if
\(
 \max_{\Gamma\in \{10,14\}}|M_\Gamma(g,n)|>\Ecal(g,n),
\)
then \(a_{g,n}\ne0\) or \(b_{g,n}\ne0\).
\end{prop}
\begin{proof}
Apply Lemma \ref{lem:term-bound} to every term not included in \eqref{eq:M-leading}, sum the resulting bounds, and use the reverse triangle inequality.
\end{proof}

\subsection{Analytic bounds and proof of Theorem \ref{thm:main nonvanishing}}

The following bounds verify the criterion of Proposition~\ref{prop:criterion} for all $g\geq2000$ and $n\geq0$.
\begin{prop}\label{prop:M lower}
For each $g\geq 2000$ and $n\geq 0$ we have that 
\begin{equation}\label{equ:M lower}
\max_{\Gamma\in \{10,14\}}|M_\Gamma(g,n)| \geq \frac 15.
\end{equation}
\end{prop}
\begin{prop}\label{prop:E upper}
For each $g\geq 2000$ and $n\geq 0$ we have that 
\begin{equation}\label{equ:E upper}
\Ecal(g,n) \leq \frac 1{30}.
\end{equation}
\end{prop}

Theorem \ref{thm:main nonvanishing} now follows from these propositions and Proposition \ref{prop:criterion}.
\begin{proof}[Proof of Theorem \ref{thm:main nonvanishing}]
The region $g\leq 2000$ is covered by Theorem \ref{thm:mainA}.
By the above propositions we have that for any $g\geq 2000$, $n\geq 0$
\[
\max_{\Gamma\in \{10,14\}}|M_\Gamma(g,n)| - \Ecal(g,n)
\geq \frac 15 - \frac 1{30} =\frac16 > 0.
\]
Hence by Proposition \ref{prop:criterion} it follows that \(a_{g,n}\ne0\) or \(b_{g,n}\ne0\) for any $g\geq 2000$, $n\geq 0$.
\end{proof}

It remains to prove Propositions~\ref{prop:M lower} and~\ref{prop:E upper}.

\subsection{Proof of Proposition \ref{prop:E upper} }
\label{subsec:error-proof}
To show the error bound \eqref{equ:E upper}, we use the following auxiliary estimates, whose proofs are deferred to Section \ref{subsec:auxiliary-estimates}:
\begin{equation}\label{eq:auxiliary-summary}
 \|\widehat H_q\|_{1,\le14}<600\quad(q\ge1),
 \qquad
 r_2=\frac{55}{12},
 \qquad
 r_m\le6^m\Gamma\left(\frac m2+1\right)\quad(m\ge3).
\end{equation}
For $2\le m\le1997$, $g\ge2000$, and $n\ge0$,
\[
 (g-2)\fall{g+n-3}{m-1}
 \ge1998\fall{1997}{m-1}.
\]
If $1998\le m\le g-3$, then
\[
 (g-2)\fall{g+n-3}{m-1}\ge(m+1)!.
\]
Finally, $\Sfac_{g,n}\ge\Sfac_g=(g-2)!/(2\pi)^g$.  Substituting these bounds into \eqref{eq:E-error} gives
\begin{equation}
\label{eq:E-majorized}
\begin{aligned}
 \Ecal(g,n)\le{}&
 \frac{600(55/12)(2\pi)^2}{1998\cdot1997}\\
 &+\sum_{m=3}^{1997}
 \frac{600\,6^m\Gamma(m/2+1)(2\pi)^m}
 {1998\fall{1997}{m-1}}\\
 &+\sum_{m=1998}^{\infty}
 \frac{600\,6^m\Gamma(m/2+1)(2\pi)^m}{(m+1)!}\\
 &+\sup_{h\ge2000}
 \frac{\bigl(6^{h-2}\Gamma(h/2)
 +2\cdot6^{h-1}\Gamma((h+1)/2)\bigr)(2\pi)^h}{(h-2)!}.
\end{aligned}
\end{equation}
The first two lines form a finite sum.  We now reduce the remaining infinite series and supremum to finitely many quantities.

The quotient of the terms with indices $m+2$ and $m$ in the infinite series is
\[
 \frac{72\pi^2}{m+3}.
\]
It decreases with $m$.  Splitting the series into its even and odd subsequences gives
\begin{align}
 &\sum_{m=1998}^{\infty}
 \frac{600\,6^m\Gamma(m/2+1)(2\pi)^m}{(m+1)!}\notag\\
 &\qquad\le
 \frac{600\,6^{1998}\Gamma(1000)(2\pi)^{1998}}
 {1999!\left(1-72\pi^2/2001\right)}
 +\frac{600\,6^{1999}\Gamma(2001/2)(2\pi)^{1999}}
 {2000!\left(1-72\pi^2/2002\right)}.
 \label{eq:series-geometric}
\end{align}

The two summands in the last line of \eqref{eq:E-majorized} have two-step ratios
\begin{equation}\label{eq:boundary-ratios}
 \frac{72\pi^2}{h-1},
 \qquad
 \frac{72\pi^2(h+1)}{h(h-1)}.
\end{equation}
Both are smaller than one and decreasing for $h\ge2000$.  Hence each parity subsequence is decreasing, and the supremum in \eqref{eq:E-majorized} is the maximum of the two values obtained at $h=2000$ and $h=2001$.
We evaluate these expressions by computer:

\begin{inputbox}
\begin{computerverification}
[\texttt{error\_finite.py}]
\label{input:error-bound}
The checker verifies
\begin{align*}
 &\frac{600(55/12)(2\pi)^2}{1998\cdot1997}
 +\sum_{m=3}^{1997}
 \frac{600\,6^m\Gamma(m/2+1)(2\pi)^m}
 {1998\fall{1997}{m-1}}
 <\frac{33}{1000}, %
 \\
 &\frac{600\,6^{1998}\Gamma(1000)(2\pi)^{1998}}
 {1999!\left(1-72\pi^2/2001\right)}
 +\frac{600\,6^{1999}\Gamma(2001/2)(2\pi)^{1999}}
 {2000!\left(1-72\pi^2/2002\right)}
 <10^{-9},%
\end{align*}
and the maximum of the two explicit boundary expressions at $h=2000,2001$ is smaller than $377/10^{13}$. It follows that
\[
\frac{33}{1000}+10^{-9}
 +\frac{377}{10^{13}}<\frac1{30}.
\]
\end{computerverification}
\end{inputbox}

Combining \eqref{eq:E-majorized}, \eqref{eq:series-geometric}, the monotonicity following \eqref{eq:boundary-ratios}, and Computer Verification \ref{input:error-bound} proves Proposition \ref{prop:E upper}.
\hfill\qed

\subsection{Interval arithmetic}
Interval arithmetic is a computational tool that allows for proving that the values of a scalar function $f(x)$ are contained in an interval $J$ if the variable $x$ varies over a compact interval $I$, or a product of such in higher dimensions. We give a short overview here and refer to the literature for more details, see \cite[§5]{Rump2010} or \cite{MKCinterval}.
Let $\mathcal I$ be the set of compact intervals $[a,b]$ in $\R$. Then we may extend the standard arithmetic operations to $\mathcal I$ by saying that the operation $\star$ (addition, subtraction, multiplication, division) on $I,J\in \mathcal I$ is 
\[
I \star J = \{x \star y \mid x\in I, y\in J\}.
\]
In the case that $\star$ is division, we additionally require that $0\notin J$ to make the operation well-defined.
In all cases the endpoints of the resulting interval $I \star J$ are elementary functions of the endpoints of $I$ and $J$, for example
\[
[a,b] + [c,d] = [a+c, b+d].
\]
Interval arithmetic is carried out by a computer program, the \emph{interval checker}, for which we use Arb \cite{Johansson2017arb}. 
Here one has to mind that if the endpoints have to be rounded to fit into finite precision numerical representations, then this rounding is always done outwards.
Let $f(x)$ be some rational arithmetic expression in the variable $x$. Then if the interval checker asserts that $f(I)\subset J$ for some compact intervals $I$ and $J$, then we have proven that $f(x)\in J$ for all $x\in I\subset \R$.
Several remarks are in order:
\begin{itemize}
    \item Note that interval arithmetic does not follow the usual computational rules, for example $x-x=0$ for all $x\in \R$, but $[0,1]-[0,1]=[-1,1]$. Hence strictly speaking the interval checker is not applied to the function $f$, but to an underlying expression tree encoding $f$.
    \item The estimate obtained by an interval checker can be improved by subdividing the interval. For example, subdividing $[0,1]$ into $n$ subintervals $[j/n,(j+1)/n]$ the expression $x-x$ would be evaluated by the interval checker to $[-1/n,1/n]$, thus improving our previous estimate $[-1,1]$. Below we shall indicate how fine we need to make this subdivision to attain useful estimates whenever we use an interval checker.
    \item Constants $c$ that can be represented exactly numerically can just be handled as zero-length intervals $[c,c]$. Other constants such as $\pi$ have to be enclosed in an interval of positive size to allow for numerical evaluation. This is handled by the program internally, we will hide this replacement in formulas.
    \item Non-rational functions like $\exp(-)$ or $\log(-)$ can be handled by sandwiching them between rational approximations, valid on the parameter range we consider.
\end{itemize}

The main strategy in the proof of Proposition \ref{prop:M lower} is to rewrite the leading order term $M_\Gamma(g,n)$ in a way that makes it amenable to be probed by an interval checker and proven to be nonzero, using variables substituting $g,n$ that vary over compact intervals.

We will use interval arithmetic not only for numerical programs, but also to encode analytic estimates on the objects we consider. It is easiest to introduce the relevant notation in one (albeit trivial) example. 
Suppose we consider a function $f(x)$, for example $f(x)=\sin(x)$, for $x\in [0, \pi/2]$. Then we may make the estimates $2x/\pi\leq \sin(x) \leq x$, and encode that estimate in an interval expression 
\[
\mathbf f(x) =\Bigl[\frac{2}{\pi} x, x\Bigr].
\]
Then $f(x)\in \mathbf f(x)$ for all $x\in [0,\pi/2]$.
Generally, we will follow the convention of indicating such interval estimates by using boldface letters, like $\mathbf f$ for the interval estimate of $f$ above. Note that here we might again insert an interval for $x$, to get an estimate for the range of $f$, for example 
\[
\mathbf f\Bigl(\Bigl[\frac{1}{3}, \frac{1}{2}\Bigr]\Bigr) = \Bigl[\frac{2}{3\pi}, \frac{1}{2}\Bigr],
\]
and we may thus conclude that $f(x)\in [2/(3\pi),1/2]$ for all $x\in [1/3,1/2]$. The latter computation could have been done numerically by our interval checker, apart from the replacement of $\pi$ by a rational interval enclosing it.

We also use similar notation for encoding estimates for the coefficients of polynomials. An interval polynomial in some variable $w$ is simply an expression 
\[
\mathbf p(w) = \sum_{k=0}^n I_k w^k,
\]
with $I_k$ intervals. For some (ordinary) polynomial $p(w)=\sum_{j=0}^m p_j w^j$ in $w$ we will use the notation 
\[
 p(w)\incoeff \mathbf p(w)=\sum_{k=0}^{n}I_kw^k
\]
to mean that $p_k=[w^k]p(w)\in I_k$ for all $k$.
The interval arithmetic and polynomial arithmetic naturally combine, i.e., we can naturally add and multiply interval polynomials. Furthermore, the endpoints of the intervals in an interval polynomial could depend on other variables, and these variables could again be evaluated on intervals by our interval checker, as will happen below.
    
\subsection{Proof of Proposition \ref{prop:M lower}}
To obtain \eqref{equ:M lower} we start from the exact expression \eqref{eq:M-leading} of the leading term, whose magnitude we want to bound from below, i.e.,
\begin{align*}
 M_\Gamma(g,n)={}&\Trunc{\Gamma}\bigl(P_{g,n}\widehat H_{g-1}\bigr)
 +\omega_{1,0}\Trunc{\Gamma}\bigl(R_{1,0}P_{g-1,n}\widehat H_{g-2}\bigr)\\
 &+\mathbf 1_{n\ge1}\,\omega_{1,1}\Trunc{\Gamma}\bigl(R_{1,1}P_{g-1,n-1}\widehat H_{g-2}\bigr),
\end{align*}
with $R_{1,0}=-1+w/2+w^2/2$ and $R_{1,1}=w/2-w^2/2$ as in \eqref{eq:R-low}.
We next replace $g$ and $n$ by the variables (depending on $g,n$)
\begin{equation}\label{eq:x-s-L}
 x:=\frac1g,
 \qquad
 s:=\frac g{g+n},
 \qquad\mbox{and}\qquad
 L:=-\log s=\log\left(1+\frac ng\right).
\end{equation}
Thus
\[
 0\le x\le\frac1{2000},
 \qquad 0<s\le1,
 \qquad L\ge0,
 \qquad s=e^{-L}.
\]
Our goal is to do this replacement, and then (essentially) just run an interval checker on the resulting function to bound $M_\Gamma(g,n)$, using compact ranges of our variables.  
In terms of these variables
\begin{align*}
\omega_{1,0}(g,n) &= 2\pi \frac{g-3}{g-2} \frac{1}{g+n-3}
=2\pi \frac{1-3x}{1-2x} \frac{xs}{1-3xs}
=2 \pi xs \frac{1-3x}{(1-2x)(1-3xs)}%
\\
\omega_{1,1}(g,n) &=2\pi \frac{g-3}{g-2} \frac{n}{(g+n-3)(g+n-4)}
=
2\pi \frac{1-3x}{1-2x} \frac{xs (1-s)}{(1-3xs)(1-4xs)}
\\&=
2\pi xs \frac{(1-3x)(1-s)}{(1-2x)(1-3xs)(1-4xs)}%
\end{align*}
Next consider 
$$
P_{g,n}=\frac{\rise{g-2+w}{n}}{\rise{g-2}{n}}
=\prod_{k=0}^{n-1}\left(1+\frac{w}{g-2+k}\right).
$$
Here the problem is that since the number of factors depends on $n$, we cannot readily replace $n$ by our compact-range variables, or rather this would not yield an elementary function that the interval checker can probe.
Rather, we need a more complicated estimate of the coefficient ranges of the polynomial $P_{g,n}$ as an interval polynomial, and likewise of $P_{g-1,n}$ and $P_{g-1,n-1}$. 

\begin{lem}\label{lem:P-estimate}
Let $(m,j)$ be one of $(0,0),(1,0),(1,1)$, with $n\ge j$. Define $\boldsymbol\beta_{m,j,q}$ by
\begin{align}
 \boldsymbol\beta_{m,j,1}:={}&
 \log(1-(m+j+2)xs)-\log(1-(m+2)x)\notag\\
 &+\left[0,
 \max\left\{0,\frac{x}{1-(m+2)x}
 -\frac{xs}{1-(m+j+2)xs}\right\}\right],
 \label{eq:beta1}\\
 \boldsymbol\beta_{m,j,q}:={}&
 \frac{(-1)^{q+1}}q\Biggl(
 \frac{
 \left(\frac{x}{1-(m+2)x}\right)^{q-1}
 -\left(\frac{xs}{1-(m+j+2)xs}\right)^{q-1}}
 {q-1}\notag\\
 &\hspace{6mm}
 +\left[0, \max\left\{0,
 \left(\frac{x}{1-(m+2)x}\right)^q
 -\left(\frac{xs}{1-(m+j+2)xs}\right)^q\right\}
 \right]\Biggr)
 \quad(2\le q\le14).
 \label{eq:betaq}
\end{align}
The interval endpoints depend on $x$ and $s$. Set $\mathbf v_{m,j,0}=[1,1]$ and define the coefficients $\mathbf v_{m,j,k}$ for $1\le k\le14$ and the polynomial $\mathbf V_{m,j}$ by
\begin{equation}\label{eq:V-recurrence}
 k\mathbf v_{m,j,k}
 =\sum_{q=1}^{k}q\boldsymbol\beta_{m,j,q}\mathbf v_{m,j,k-q}
 \qquad\mbox{and}\qquad
 \mathbf V_{m,j}(x,s,w)
 :=\sum_{k=0}^{14}\mathbf v_{m,j,k}w^k.
\end{equation}
Then
\begin{equation}\label{eq:P-estimate}
 P_{g-m,n-j}(w)
 \incoeff e^{Lw}\mathbf V_{m,j}(x,s,w)
 \pmod{w^{15}}.
\end{equation}
\end{lem}
The proof is deferred to Subsection \ref{subsubsec:P-proof}.
We note that applying $\max\{0,\cdots\}$ in the definition of the intervals is strictly speaking not necessary for the statement of the lemma, but we would like the intervals to remain well-defined even for parameter values not covered by the lemma, for which the upper bound might become negative.
Similarly, consider the polynomials $\widehat H_{g-1}(w)$ and $\widehat H_{g-2}(w)$. Here the problem is similar in that the $g$-dependence does not allow for a direct replacement of $g$ by $1/x$ to get an elementary function that the interval checker can work with. We instead proceed as for $P_{g,n}$ and estimate the first 15 coefficients of $\widehat H_h(w)$ by an interval polynomial.
To this end let $\sigma\in\Z/4\Z$ and define the real polynomial
\begin{equation}\label{eq:Lsigma}
 \mathscr L_\sigma(w)
 :=i^{-(\sigma+1)}\sum_{k=1}^{14}
 \frac{(-2\pi i)^k+(-1)^{\sigma+1}(2\pi i)^k}{k!}w^k.
\end{equation}
Only one parity of $k$ survives, depending on the parity of $\sigma$, so the coefficients are real.

\begin{lem}\label{lem:H-estimate}
Let $\rho=g\bmod4$ and $m\in\{0,1\}$.  Define
\begin{align}
 \mathbf H_{\rho,m}(x,w):={}&
 \mathscr L_{\rho-1-m}(w)
 +\frac{\pi x}{1-(m+1)x}(w-w^2)
  \mathscr L_{\rho-2-m}(w)\notag\\
 &+\sum_{k=1}^{14}
 \left[-\frac{13000x^2}{(1-(m+1)x)^2},
        \frac{13000x^2}{(1-(m+1)x)^2}\right]w^k,
 \label{eq:H-box}
\end{align}
where the subscripts of $\mathscr L$ are read modulo four.  Then, for $g\ge2000$,
\begin{equation}\label{eq:H-estimate}
 \widehat H_{g-1-m}(w)
 \incoeff\mathbf H_{\rho,m}(x,w)
 \pmod{w^{15}}.
\end{equation}
Moreover,
\begin{equation}\label{eq:H-linear-parity}
 [w]\widehat H_q(w)=0
 \qquad(q>1\text{ odd}).
\end{equation}
Whenever \eqref{eq:H-linear-parity} applies, the coefficient interval of $w$ in \eqref{eq:H-box} is replaced by $[0,0]$.
\end{lem}

The proof is deferred to Subsection \ref{subsubsec:H-proof} below.
Substituting Lemmas \ref{lem:P-estimate} and \ref{lem:H-estimate} into \eqref{eq:M-leading}, with $\mathbf V_{m,j}$ as in \eqref{eq:V-recurrence}, gives the interval expression
{\small
\begin{equation}\label{eq:Mcal}
\begin{aligned}
 &\Mcal_{\Gamma,\rho}(x,s,L):=
 \Trunc{\Gamma}\!\bigl(
 e^{Lw}\mathbf V_{0,0}(x,s,w)\mathbf H_{\rho,0}(x,w)
 \bigr)\\
 &\quad\quad\quad\quad\quad\quad +2\pi xs\frac{1-3x}{(1-2x)(1-3xs)}
 \Trunc{\Gamma}\!\bigl(
 R_{1,0}(w)e^{Lw}\mathbf V_{1,0}(x,s,w)
 \mathbf H_{\rho,1}(x,w)
 \bigr)\\
 &\quad\quad\quad\quad\quad\quad +2\pi xs\frac{(1-3x)(1-s)}
 {(1-2x)(1-3xs)(1-4xs)}
 \Trunc{\Gamma}\!\bigl(
 R_{1,1}(w)e^{Lw}\mathbf V_{1,1}(x,s,w)
 \mathbf H_{\rho,1}(x,w)
 \bigr).
\end{aligned}
\end{equation}}
For every actual pair $g\ge2000$, $n\ge1$, the natural coefficientwise interval evaluation of \eqref{eq:Mcal} contains $M_\Gamma(g,n)$.
Note that for $n=1$ in the last term one has $P_{g-1,0}=1$, and the interval in Lemma \ref{lem:P-estimate} collapses to a point.
The same expression also includes $n=0$, represented by the endpoint $(s,L)=(1,0)$. This case should be argued separately, since Lemma \ref{lem:P-estimate} requires $n\geq j$. The first two summands of \eqref{eq:Mcal} are still covered by Lemma \ref{lem:P-estimate} with $n=j=0$.  The third summand is absent from \eqref{eq:M-leading}, while its prefactor in \eqref{eq:Mcal} contains $1-s$ and is therefore exactly zero at this endpoint. Thus all $n\ge0$ are covered by the one compact interval formula  \eqref{eq:Mcal}.

\begin{lem}\label{lem:L-degree}
For fixed $x,s$, the expression $\Mcal_{\Gamma,\rho}(x,s,L)$ is a polynomial in $L$ of degree at most $\Gamma-1$.
\end{lem}

\begin{proof}
The only $L$-dependence is in $e^{Lw}$ and the factors $\mathbf H_{\rho,m}$ ($m=0,1$) are divisible by $w$. Since we apply the operator $\Trunc{\Gamma}$ we see that $\Mcal_{\Gamma,\rho}(x,s,L)$ is a polynomial in $L$ of degree at most $\Gamma-1$.
\end{proof}

Although $\rho=g\bmod4$ occurs in \eqref{eq:Mcal}, only parity matters for the absolute-value estimate.  Indeed,
\[
 \mathscr L_{\sigma+2}=-\mathscr L_\sigma,
\]
and the error intervals in \eqref{eq:H-box} are symmetric about zero.  Hence, as interval sets,
\begin{equation}\label{eq:rho-symmetry}
 \mathbf H_{\rho+2,m}=-\mathbf H_{\rho,m},
 \qquad
 \Mcal_{\Gamma,\rho+2}=-\Mcal_{\Gamma,\rho}.
\end{equation}
It is therefore enough to consider one even and one odd residue class.

We now evaluate \eqref{eq:Mcal}.  The only unbounded variable is $L$, so we divide the proof into
\[
 \text{Case A: }0\le L\le32,
 \qquad
 \text{Case B: }L\ge32.
\]

\subsubsection{Evaluation in Case A: $0\le L\le32$}
\label{subsubsec:case-A}

On this range we directly evaluate the full expression \eqref{eq:Mcal} using the interval checker.  Mathematically, one could use the single box
\[
 0\le x\le\frac1{2000},
 \qquad
 0\le s\le1,
 \qquad
 0\le L\le32.
\]
But the interval evaluation on such a large box is too coarse, so the program subdivides the $x$- and $L$-intervals.  It uses the 256 equal intervals
\[
 \left[\frac r{512000},\frac{r+1}{512000}\right],
 \qquad 0\le r<256,
\]
for $x$, and subdivisions of $[0,32]$ into $4208$ intervals in the even case and $4192$ intervals in the odd case.  For an $L$-interval $J=[L^-,L^+]$, it encloses the actual value $s=e^{-L}$ by a rational interval containing
\[
 [e^{-L^+},e^{-L^-}].
\]
It then evaluates $\Mcal_{10,\rho}$ and $\Mcal_{14,\rho}$ on the resulting product boxes.  Treating $s$ and $L$ as independent inside these containing boxes can at worst only enlarge the interval and is therefore rigorous.  The endpoint $(s,L)=(1,0)$ is included in the same evaluation.

On every box, and for one even and one odd value of $\rho$, the checker finds $\Gamma\in\{10,14\}$ and a sign $\varepsilon\in\{-1,1\}$ such that
\begin{equation}\label{eq:case-A-check}
 \varepsilon\,\Mcal_{\Gamma,\rho}(x,s,L)>\frac15
\end{equation}
throughout the box.  The symmetry \eqref{eq:rho-symmetry} covers the other two residue classes.

\subsubsection{Evaluation in Case B: $L\ge32$}
\label{subsubsec:case-B}

In this range the choice $\Gamma=10$ already suffices.  Put
\[
 \vartheta:=\frac1L.
\]
Then
\[
 0<\vartheta\le\frac1{32},
 \qquad
 0<s=e^{-L}\le e^{-32}.
\]
We include $\vartheta=0$ only to compactify the interval boxes.  By Lemma \ref{lem:L-degree}, $\Mcal_{10,\rho}$ has degree at most $9$ in $L$. 
For odd $g$, the interval checker evaluates
\[
 \vartheta^9\Mcal_{10,\rho}(x,s,1/\vartheta)
\]
on the compact box
\[
 0\le x\le\frac1{2000},
 \qquad
 0\le s\le e^{-32},
 \qquad
 0\le\vartheta\le\frac1{32},
\]
and verifies, for a fixed sign $\varepsilon\in\{-1,1\}$,
\begin{equation}\label{eq:odd-case-B-check}
 \varepsilon\,\vartheta^9
 \Mcal_{10,\rho}(x,s,1/\vartheta)
 -\frac15\vartheta^9>0.
\end{equation}
At every point $\vartheta>0$, division by $\vartheta^9$ gives $|M_{10}(g,n)|>1/5$.

For even $g$, the coefficient of $L^9$ is exponentially suppressed.  A term of degree $9$ in $L$ can only come from the coefficient of $w$ in the factor multiplying $e^{Lw}$.  In the first summand of \eqref{eq:Mcal}, that coefficient vanishes by \eqref{eq:H-linear-parity}, because $g-1$ is odd.  The third summand contributes nothing because $R_{1,1}(0)=0$.  Hence
\[
 [L^9]\Mcal_{10,\rho}(x,s,L)
 =sJ_\rho(x,s),
\]
where
\[
 J_\rho(x,s)
 :=-\frac{2\pi x(1-3x)}{9!(1-2x)(1-3xs)}
 [w]\mathbf H_{\rho,1}(x,w).
\]
Write
\[
 \Mcal_{10,\rho}(x,s,L)
 =A_\rho(x,s,L)+sJ_\rho(x,s)L^9,
 \qquad
 \deg_L A_\rho\le8.
\]

Multiply by $\vartheta^8$ and use again $L=1/\vartheta$ to obtain 
\begin{equation}\label{eq:even-reversal}
 \vartheta^8\Mcal_{10,\rho}(x,s,1/\vartheta)
 =\vartheta^8A_\rho(x,s,1/\vartheta)
 +\kappa J_\rho(x,s),
\end{equation}
with
\[
 \kappa:=Ls.
\]
Note that $\kappa$ is a bounded quantity as $Ls=Le^{-L}$:
Since $Le^{-L}$ decreases for $L\ge1$,
\[
 0\le\kappa\le32e^{-32}.
\]
The right-hand side of \eqref{eq:even-reversal} is regular at $\vartheta=0$.  The checker verifies on the compact box in $x,s,\vartheta,\kappa$ that, for a fixed sign $\varepsilon\in\{-1,1\}$,
\begin{equation}\label{eq:even-case-B-check}
 \varepsilon\left(
 \vartheta^8A_\rho(x,s,1/\vartheta)
 +\kappa J_\rho(x,s)\right)
 -\frac15\vartheta^8>0.
\end{equation}
Again, division by $\vartheta^8$ at every actual point gives $|M_{10}(g,n)|>1/5$.

Consequently,
\begin{equation}\label{eq:case-B-conclusion}
 L\ge32\quad\Longrightarrow\quad |M_{10}(g,n)|>\frac15.
\end{equation}

\begin{inputbox}
\begin{computerverification}
[\texttt{interval\_check\_m.py}]
\label{input:leading-evaluation}
The interval checker performs interval evaluations of the explicit expression \eqref{eq:Mcal} as described above.  It verifies:
\begin{enumerate}[label=(\roman*)]
 \item on $0\le L\le32$, including the endpoint $n=0$, all $2{,}150{,}400$ product boxes described above satisfy \eqref{eq:case-A-check} for one of $\Gamma=10,14$;
 \item on $L\ge32$, the choice $\Gamma=10$ alone suffices: the odd compact box satisfies \eqref{eq:odd-case-B-check}, and the even compact box satisfies \eqref{eq:even-case-B-check}.
\end{enumerate}
The input expressions may contain $\pi$, logarithms, and exponentials.  The checker first replaces every such standard quantity by a certified rational interval, obtained from a finite Taylor expansion with an explicit remainder, and then performs the arithmetic outward.  %
\end{computerverification}
\end{inputbox}

\begin{proof}[Proof of Proposition \ref{prop:M lower}]
Case A follows from \eqref{eq:case-A-check}, which is verified in Computer Verification~\ref{input:leading-evaluation}(i).  In Case B, \eqref{eq:odd-case-B-check} and \eqref{eq:even-case-B-check}, verified in Computer Verification~\ref{input:leading-evaluation}(ii), prove \eqref{eq:case-B-conclusion}; the actual value of $\vartheta$ is positive, so the displayed powers of $\vartheta$ may be divided out.  By \eqref{eq:rho-symmetry}, only two values of $\rho$ need to be checked.
\end{proof}

\begin{rem}[Critical slopes]\label{rem:critical-slopes}
The following heuristic explains why the proof of Proposition~\ref{prop:M lower} works, and why the cases $\Gamma=10$ and $\Gamma=14$ must be treated together. Fix the ratio $n/g$ (the \emph{slope}) and let $g\to\infty$. In the coordinates \eqref{eq:x-s-L}, $x\to0$, while $L=\log(1+n/g)$ and $s=e^{-L}$ remain constant, depending only on the slope. In \eqref{eq:Mcal}, only the first term survives at $x=0$, and since $s$ occurs only in products with $x$, the dependence on $s$ disappears as well. What remains is
\[
\Trunc{\Gamma}\!\bigl(
 e^{Lw}\mathscr L_{\rho-1}(w)
 \bigr).
\]
This is a polynomial in $L$ whose coefficients depend on $\rho=g\bmod 4$. Its zeros determine the \emph{critical slopes}, at which the leading term vanishes in the limit. Figure~\ref{fig:all-critical-slopes} displays numerical approximations to all its nonnegative real zeros for $\Gamma=10,14$ and each parity of $g$. A marker at $L$ corresponds to the critical ray $n=(e^L-1)g$. Along a critical slope the nonvanishing of $c_{g,n}^{(\Gamma)}$ can only come from sub-leading contributions, which are of order $x=1/g$, and the leading-term argument gives no information. Away from the critical slopes, the leading term alone establishes nonvanishing. The point is that, for each parity of $g$, the polynomials for $\Gamma=10$ and $\Gamma=14$ have no common zero, so their critical slopes are disjoint. Consequently, for every slope at least one of the two leading terms is bounded away from zero as $x\to0$, and the nonvanishing of the corresponding coefficient, $a_{g,n}$ or $b_{g,n}$, follows without any analysis of sub-leading contributions.
\end{rem}

\begin{figure}[tbp]
\centering
\begin{tikzpicture}[font=\small,x=.47cm,y=.85cm]
  \node[anchor=west,font=\small\scshape] at (0,4.7) {Even genus};
  \node[anchor=west,font=\small\scshape] at (0,1.5) {Odd genus};
  \foreach \yy/\lab in {4.0/{\chi_{11}},3.1/{\chi_{15,0}},.8/{\chi_{11}},-.1/{\chi_{15,0}}} {
    \draw[black!25] (0,\yy)--(27,\yy);
    \node[anchor=east] at (-.5,\yy) {$\lab$};
  }
  \foreach \xx in {.288994091,2.209960145,4.623784650,9.154923928} {
    \fill[weightblue] (\xx,4.0) circle[radius=2pt];
  }
  \foreach \xx in {.160280672,1.516971468,3.026172695,4.896889655,7.548952726,13.429795496} {
    \node[regular polygon,regular polygon sides=3,inner sep=1.4pt,fill=hodgeorange] at (\xx,3.1) {};
  }
  \foreach \xx in {1.220081119,3.316040419,6.310820710,17.694402641} {
    \fill[weightblue] (\xx,.8) circle[radius=2pt];
  }
  \foreach \xx in {.828836092,2.242241046,3.897311060,6.086129178,9.563971622,25.792168724} {
    \node[regular polygon,regular polygon sides=3,inner sep=1.4pt,fill=hodgeorange] at (\xx,-.1) {};
  }
  \draw (0,-.65)--(27,-.65);
  \foreach \xx in {0,3,...,27} {
    \draw (\xx,-.65)--(\xx,-.76) node[below] {$\xx$};
  }
  \node at (13.5,-1.75) {$L=\log(1+n/g)$};
\end{tikzpicture}
\caption{The disjoint sets of critical slopes for each parity of $g$.}
\label{fig:all-critical-slopes}
\end{figure}
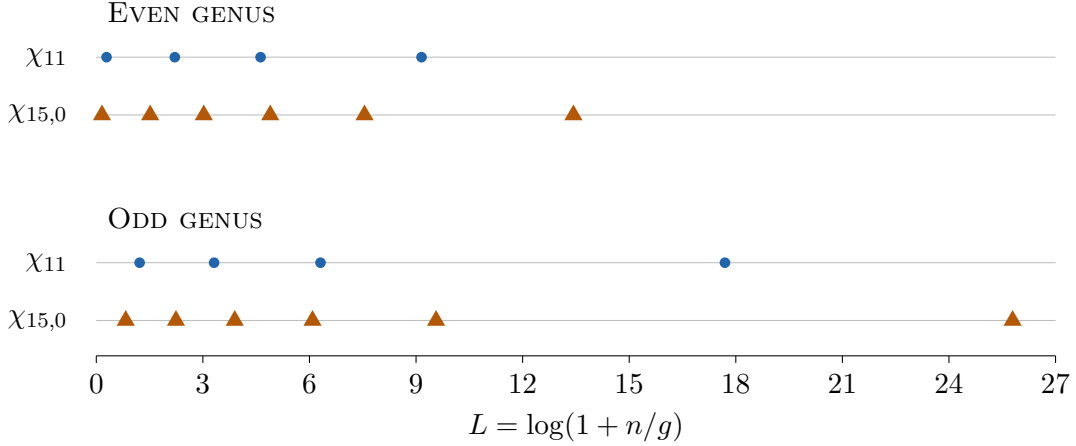


\FloatBarrier
\section{Proofs of the auxiliary estimates}
\label{subsec:auxiliary-estimates}

It remains to prove Lemmas \ref{lem:P-estimate} and \ref{lem:H-estimate} and the estimates \eqref{eq:auxiliary-summary}.

\subsection{Proof of Lemma \ref{lem:P-estimate}}
\label{subsubsec:P-proof}
Let $g_0=g-m$, $n_0=n-j$, and $a=g_0-2$.  From
\[
 P_{g_0,n_0}(w)
 =\prod_{v=0}^{n_0-1}\left(1+\frac{w}{a+v}\right)
\]
we obtain, modulo $w^{15}$,
\begin{equation}\label{eq:log-P}
 \log P_{g_0,n_0}(w)
 =\sum_{q=1}^{14}\frac{(-1)^{q+1}}q
 \left(\sum_{v=0}^{n_0-1}(a+v)^{-q}\right)w^q.
\end{equation}
For a positive decreasing function $f$,
\begin{align*}
 0&\le\sum_{v=0}^{n_0-1}f(a+v)-\int_a^{a+n_0}f(t)\,dt
 =
 \sum_{v=0}^{n_0-1}
\left(
f(a+v)
-
\int_{a+v}^{a+v+1}f(t)\,dt
 \right)
 \\&\le 
  \sum_{v=0}^{n_0-1}
\left(
f(a+v)
-
f(a+v+1)
 \right)
 = f(a)-f(a+n_0).
\end{align*}
Applying this to $f(t)=t^{-1}$ yields
\begin{align}
 \log\frac{a+n_0}{a}
 &\le\sum_{v=0}^{n_0-1}\frac1{a+v}
 \le\log\frac{a+n_0}{a}+\frac1a-\frac1{a+n_0},\label{eq:harmonic-integral-1}
\end{align}
and to $f(t)=t^{-q}$ with $q\geq 2$ yields 
\begin{align}
 \frac{a^{1-q}-(a+n_0)^{1-q}}{q-1}
 &\le\sum_{v=0}^{n_0-1}(a+v)^{-q}\notag\\
 &\le\frac{a^{1-q}-(a+n_0)^{1-q}}{q-1}
 +a^{-q}-(a+n_0)^{-q}
 \quad(2\le q\le14).\label{eq:harmonic-integral-q}
\end{align}
These inequalities are exact at $n_0=0$.

For the three pairs $(m,j)$ under consideration,
\[
 \frac1a=\frac{x}{1-(m+2)x},
 \qquad
 \frac1{a+n_0}=\frac{xs}{1-(m+j+2)xs},
\]
and
\[
 \log\frac{a+n_0}{a}
 =L+\log(1-(m+j+2)xs)-\log(1-(m+2)x).
\]
Subtracting $Lw$ from \eqref{eq:log-P} and using \eqref{eq:harmonic-integral-1}--\eqref{eq:harmonic-integral-q} gives precisely the intervals \eqref{eq:beta1}--\eqref{eq:betaq}. 
The following lemma is elementary and well-known. 
\begin{lem}\label{lem:exp recursion}
Let $\sum_{q\geq 1} b_q w^q$ be a formal power series and denote by $a_k$ the Taylor coefficients of its exponential,
\[
\exp\left( \sum_{q\geq 1} b_q w^q \right)
=
\sum_{k=0}^\infty a_k w^k.
\]
Then we have $a_0=1$ and for each $k\geq 1$ the recursion relation 
\[
k a_k = \sum_{q=1}^k q b_q a_{k-q}.
\]
\end{lem}
\begin{proof}
Evaluation at $w=0$ yields $a_0=1$.
Differentiation and substitution of the defining equation gives
    \[
\sum_{k\geq 1} k a_k w^{k-1}
=
 \left(\sum_{q\geq 1}q b_q w^{q-1}\right)\exp\left(\sum_{q\geq 1}b_q w^q\right)
 =
 \left(\sum_{q\geq 1} q b_q w^{q-1}\right) \left( \sum_{k\geq 0} a_k w^k\right).
\]
Comparing the coefficients of $w^{k-1}$ on either side
yields the recursion relation.
\end{proof}

We apply the lemma to
\[
 \exp\left(\sum_{q=1}^{14}\beta_qw^q\right)
 =\sum_{k=0}^{14}v_kw^k\pmod{w^{15}},
\]
and obtain the recursion (with $v_0=1$)
\[
 kv_k=\sum_{q=1}^kq\beta_qv_{k-q}.
\]
The interval recurrence \eqref{eq:V-recurrence} therefore encloses every coefficient, which proves \eqref{eq:P-estimate} and hence Lemma \ref{lem:P-estimate}.\hfill\qed

\subsection{Recurrence and Fourier estimates}
We establish estimates for the proof of Lemma \ref{lem:H-estimate} and the uniform norm bound in \eqref{eq:auxiliary-summary}.
Let
\begin{equation}\label{eq:dq-ellq}
 d_q(w):=-\frac{B_{q+1}(1-w)-B_{q+1}}{q(q+1)}.
\end{equation}
We have
$\Hcal(\tau,w)=\exp(\sum_{q\ge1}d_q(w)\tau^q)=\sum_{h\geq 0} H_h(w) \tau^h$.
Applying Lemma \ref{lem:exp recursion} yields the recurrence
\[
q H_q = \sum_{r=0}^{q-1} (q-r) d_{q-r} H_r.
\]
Using the rescaled version $\widehat H_q=H_q / \Sfac_{q+1}$ this becomes, with $\ell_q := \frac{d_q(w)}{\Sfac_{q+1}}$ and $\Sfac_{q+1}=(q-1)!/(2\pi)^{q+1}$,
\begin{equation}\label{eq:H-recurrence}
 \widehat H_q
 =\ell_q+\frac1{2\pi}\sum_{r=1}^{q-1} c_{q,r}
 \ell_{q-r}\widehat H_r \quad\text{ with } c_{q,r}:=\frac{(r-1)!}{\fall{q}{r}} = \frac{(r-1)!(q-r)!}{q!}.
\end{equation}
The well-known Fourier expansion of the Bernoulli polynomials \cite[Equation (24.8.3)]{NIST:DLMF} (valid for $0<w<1$) yields
\[
 \ell_q(w)
 =i^{-(q+1)}\sum_{a\ne0}
 \frac{e^{-2\pi iaw}-1}{a^{q+1}}.
\]
The two modes $a=\pm1$, truncated after $w^{14}$, are exactly
$\mathscr L_{q\bmod4}(w)$ from \eqref{eq:Lsigma}. In the following we will just write $\mathscr L_{q}(w)$, and read the subscript $q$ modulo 4 in this expression. 
In the recurrence \eqref{eq:H-recurrence}, the term with $r=1$ equals
\[
 \frac1{2\pi q}\ell_{q-1}\widehat H_1,
 \qquad
 \widehat H_1=2\pi^2(w-w^2).
\]
Replacing $\ell_{q-1}$ by its two leading Fourier modes gives the first correction
\[
 \frac{\pi}{q}(w-w^2)\mathscr L_{q-1}(w).
\]
Thus the center of the interval in \eqref{eq:H-box} is the leading two Fourier modes together with the first term of the recurrence.

We will denote the error made by dropping the higher terms in the Fourier expansion by 
\[
F_q(w) := \ell_q(w) - \mathscr L_q(w)
.
\]
We estimate its norm for $q\geq 16$:
\[
\| F_q \|_{1,\leq 14}
\leq 
2\sum_{a=2}^{\infty}
\sum_{j=1}^{14} \frac{(2\pi a)^j}{j! a^{q+1}}
=
2
\sum_{j=1}^{14}
\frac{(2\pi)^j}{j!}
(\zeta(q+1-j)-1).
\]
We bound the Riemann $\zeta$-function (for $s\geq 3$) by 
\[
\zeta(s) -1 =2^{-s} + \sum_{a\geq 3} a^{-s}
\leq 2^{-s} + \int_{2}^\infty x^{-s} dx
=
2^{-s}+
\frac{2^{1-s}}{s-1} \leq 2^{1-s}.
\]
This yields
\begin{equation}\label{equ:F estimate}
\| F_q \|_{1,\leq 14}
\leq 
2
\sum_{j=1}^{14}
\frac{(2\pi)^j}{j!} 2^{j-q}
=
2^{-q} \cdot 2\sum_{j=1}^{14} \frac{(4\pi)^j}{j!}
< 10^6 \cdot 2^{-q}.
\end{equation}
Furthermore, $\mathscr L_q(w)$ depends on $q$ only modulo 4 and $\mathscr L_q(w)=-\mathscr L_{q+2}(w)$.
Hence the explicit evaluation \eqref{input:L-small} for $q=0,1$ in Computer Verification~\ref{input:H-estimate} shows that
\[
    \| \mathscr L_q \|_{1,\leq 14}
    < 534 \quad \text{for all $q$}.
\]
By our estimate for $F_q$ above we hence find for all $q\geq 50$
\begin{equation}\label{eq:lq large}
\| \ell_q \|_{1,\leq 14}
=
\| \mathscr L_q + F_q \|_{1,\leq 14}
\leq 
\| \mathscr L_q \|_{1,\leq 14}
+\|  F_q \|_{1,\leq 14}
< 534 + 10^6 \cdot 2^{-q} < 535.
\end{equation}
By explicit evaluation \eqref{input:lq small} for $q\leq 50$ we then conclude 
\begin{equation}\label{eq:lq small}
\| \ell_q \|_{1,\leq 14} <800 \quad \text{for all $q\geq 1$}.
\end{equation}

Next let us estimate the factors $c_{q,r}$ appearing in \eqref{eq:H-recurrence}. Note that $c_{q,r}=c_{q,q-r+1}$. Furthermore, 
\[
\frac{c_{q,r+1}}{c_{q,r}} = \frac{1}{\frac{q}{r}-1},
\]
so that one sees that $c_{q,r}$ decreases as $r=1,2,\dots$ until $r$ passes $q/2$ and then increases again. We may hence estimate the following sum by its (equal) endpoints
\[
\sum_{r=5}^{q-4}c_{q,r} \leq (q-8) c_{q,5} = (q-8) \frac{24}{\fall{q}{5}}.
\]
This yields the following estimate for the larger sum (valid for $q\geq 100$ so that we can absorb all terms into the first)
\begin{equation}\label{eq:sum cqr}
\begin{aligned}
\sum_{r=2}^{q-1}c_{q,r}  &\leq 
 \frac{2}{\fall{q}{2}}
 +  \frac{4}{\fall{q}{3}}
 +  \frac{12}{\fall{q}{4}}
+ (q-8) \frac{24}{\fall{q}{5}}
\\&=
\frac{1}{q (q-1)}
\left( 2 + \frac{4}{q-2} +\frac{12}{(q-2)(q-3)}
+ \frac{24(q-8)}{(q-2)(q-3)(q-4)} \right)
\\&\leq \frac{3}{q (q-1)},
\end{aligned}
\end{equation}
Similarly, we also record for later use the inequality 
\begin{equation}\label{eq:sum cqr 3}
\begin{aligned}
\sum_{r=3}^{q-2}c_{q,r}  &\leq 
\frac{4}{\fall{q}{3}}
 +  \frac{12}{\fall{q}{4}}
+ (q-8) \frac{24}{\fall{q}{5}}
\\&=
\frac{1}{q (q-1)(q-2)}
\left( 4 +\frac{12}{(q-3)}
+ \frac{24(q-8)}{(q-3)(q-4)} \right)
\\&\leq \frac{5}{q (q-1)(q-2)}.
\end{aligned}
\end{equation}

Let us also list the following explicit computations, used in this subsection and the following two.
\begin{inputbox}
\begin{computerverification}[\texttt{llh\_bounds.py}]
\label{input:H-estimate}
\begin{equation}\label{input:L-small}
    \| \mathscr L_q \|_{1,\leq 14}
    < 534 \quad \text{for $q=0,1$}.
\end{equation}
\begin{equation}\label{input:lq small}
\| \ell_q \|_{1,\leq 14} <800 \quad \text{for $q=1,\dots, 50$}.
\end{equation}
\begin{align}
 \|\widehat H_q\|_{1,\le14}&<600 \quad \text{for all $q=1,\dots,99$}.
 \label{input:H-norm-bound}
 \end{align}
\end{computerverification}
\end{inputbox}
\subsection{A uniform norm bound for $\widehat H_q$}
Let us now turn to showing the first part of \eqref{eq:auxiliary-summary}, i.e., that $\|\widehat H_q\|_{1,\le14}<600$ for all $q\ge1$. We do this by induction. The base cases $q=1,\dots,99$ are the bound \eqref{input:H-norm-bound} of Computer Verification~\ref{input:H-estimate}.
Then assume inductively that $\|\widehat H_r\|_{1,\le14}<600$ for all $r\leq q-1$ and $q\geq 100$. We then use \eqref{eq:H-recurrence} to estimate
\begin{align*}
\|\widehat H_q\|_{1,\le14} &\leq 
\|\ell_q\|_{1,\le14}
+
\frac{\pi}{q} \|(w-w^2) \ell_{q-1} \|_{1,\le14}
+\frac1{2\pi}\sum_{r=2}^{q-1} c_{q,r}
\|
 \ell_{q-r}\widehat H_r\|_{1,\le14}
\\
&\le 535 
+
\frac{2\pi}{q} 535
+
\frac{1}{2\pi} 800 \cdot 600
\sum_{r=2}^{q-1} c_{q,r}
\\
&\le 535 
+
\frac{2\pi}{q} 535
+ \frac{400}{\pi} \frac{3}{q(q-1)} 600
\\&\leq
535 + \frac{1070 \pi}{100} + \frac{720000}{9900\pi}
<600.
\end{align*}
Here we use the induction hypothesis, \eqref{eq:lq large}, \eqref{eq:lq small} and \eqref{eq:sum cqr}. This shows the first part of \eqref{eq:auxiliary-summary} as desired.

\subsection{Proof of Lemma \ref{lem:H-estimate}}
\label{subsubsec:H-proof}
Define the error term to be 
\[
\varepsilon_q (w) := \widehat H_q (w)-\mathscr L_q(w) - \frac{\pi}{q}(w-w^2)\mathscr L_{q-1}(w).
\]
Our goal is to check that $\|\varepsilon_q \|_{1,\leq 14} \leq 13000/q^2$ for $q\geq 1998$.
To this end, we use the recursion relation \eqref{eq:H-recurrence} and the definition of the error terms $F_q$ to get 
\[
\varepsilon_q
=
F_q + \frac{\pi}{q}(w-w^2) F_{q-1}
+
\frac{1}{2\pi}
\sum_{r=2}^{q-1} c_{q,r} \ell_{q-r} \widehat H_r
.
\]
We then estimate, treating separately also the terms with $r=2$ and $r=q-1$ in the sum,
\begin{align*}
\|\varepsilon_q\|_{1, \leq 14}
&\leq 
\|F_q\|_{1, \leq 14}
+
\frac{2\pi}{q} \|F_{q-1}\|_{1, \leq 14}
+ 
\frac{1}{2\pi} c_{q,2}
\|\ell_{q-2}\|_{1, \leq 14}
\|\widehat H_{2}\|_{1, \leq 14}
 \\
&\quad +
\frac{1}{2\pi} \sum_{r=3}^{q-2}
c_{q,r} \|\ell_{q-r}\|_{1, \leq 14}\|\widehat H_{r}\|_{1, \leq 14}
+
\frac{1}{2\pi}c_{q,q-1} \|\ell_{1}\|_{1, \leq 14}\|\widehat H_{q-1}\|_{1, \leq 14}.
\end{align*}
We may estimate the individual terms on the right as follows. The first two terms are covered by \eqref{equ:F estimate}, and for the third we may use \eqref{eq:lq large} and the explicit evaluation 
\[
\|\widehat H_{2}\|_{1, \leq 14} = \left\|\pi^3\left(\frac23w -w^2-\frac23 w^3+w^4\right) \right\|_{1,\leq 14} = \frac{10}{3}\pi^3.
\]
For the fourth term (the sum) we use \eqref{eq:sum cqr 3}, and for the last term that $\ell_1(w)=2\pi^2(w-w^2)$.
Altogether, we obtain
\begin{align*}
q^2\|\varepsilon_q\|_{1, \leq 14}
&\leq 
10^6 \cdot 2^{-q}q^2
+
 10^6 \frac{2\pi}{q} \cdot 2^{1-q} q^2
+ 
\frac{1}{2\pi} \frac{q^2}{q (q-1)}
535
\frac{10}{3}\pi^3
 \\
&\quad +
\frac{1}{2\pi} 
\frac{5q^2}{q(q-1)(q-2)}
800 \cdot 600
+
\frac{1}{2\pi}\frac{q^2}{q(q-1)} 4\pi^2
600.
\end{align*}

All terms are monotonically decreasing in $q\geq 1998$, and hence can be bounded from above by their value at $q=1998$.
We hence see that 
\[
q^2 \|\varepsilon_q\|_{1, \leq 14} \leq 13000 \quad \text{for $q\geq 1998$}.
\]
Substituting $q=g-1-m$ so that 
\[
\frac 1q = \frac{x}{1-(m+1)x}
\]
this shows \eqref{eq:H-estimate}.

Finally, since the $d_q$ of \eqref{eq:dq-ellq} have no constant term, the linear terms of $H_q$ and $d_q$ agree, 
\[
[w] H_q = [w] d_q = - \left.\frac{d}{dw}\right|_{w=0} \frac{B_{q+1}(1-w)-B_{q+1}}{q(q+1)} 
=
\frac {B_q(1)}{q}.
\]
Here we used again the derivative rule for the Bernoulli polynomials $B_{q+1}'(x)=(q+1)B_{q}(x)$.
Furthermore, it is well-known that $B_q(1)=0$ for $q>1$ odd, so that \eqref{eq:H-linear-parity}, and hence Lemma \ref{lem:H-estimate} follow. 
\hfill\qed

\subsection{Coefficient bounds for $\Rcal$}
\label{subsubsec:R-proof}

It remains to prove the bound for $r_m$ in
\eqref{eq:auxiliary-summary}.  We work throughout modulo $w^{15}$ and
write
\begin{equation}\label{eq:Lambda-def}
 \log\Rcal(u,z,w)=\sum_{m\ge1}\Lambda_m(z,w)u^m.
\end{equation}
For a polynomial in $z,w$, let $\|\cdot\|_+$ denote the sum of the
absolute values of all its coefficients, and put
\begin{equation}\label{eq:lambda-def}
 \lambda_m:=\|\Lambda_m\|_+.
\end{equation}
For $m\ge0$ set
\[
 \Rcal_m(z,w):=[u^m]\Rcal(u,z,w)
 =\sum_{j=0}^m R_{m,j}(w)z^j,
\]
so that $r_m=\|\Rcal_m\|_+$ by \eqref{eq:norm}.  By Lemma \ref{lem:exp recursion}
we have
\[
 m\Rcal_m=\sum_{\nu=1}^m \nu\Lambda_\nu \Rcal_{m-\nu}.
\]
Taking the norm $\|-\|_+$ on both sides yields the recursion relation
\begin{equation}\label{eq:r-log-recurrence}
 mr_m\le \sum_{\nu=1}^m \nu\lambda_\nu r_{m-\nu}.
\end{equation}

\begin{lem} \label{lem:central ineq}
We have that
\begin{equation}\label{eq:lambda-scalar-target}
 \sum_{m=1}^{\infty}
 \frac{m\lambda_m}{6^m\Gamma(m/2+1)}<1.
\end{equation}
\end{lem}
The proof of the lemma will be given below. For now we shall check that it immediately implies the desired bound for $r_m$ in
\eqref{eq:auxiliary-summary}.

\begin{cor}\label{cor:r-bound}
For every $m\ge0$,
\begin{equation}\label{eq:r-bound}
 r_m\le6^m\Gamma\left(\frac m2+1\right).
\end{equation}
\end{cor}
\begin{proof}
We shall use the elementary inequality
\begin{equation}\label{eq:gamma-product}
 \Gamma\left(\frac\nu2+1\right)
 \Gamma\left(\frac{m-\nu}{2}+1\right)
 \le \Gamma\left(\frac m2+1\right)
 \qquad(0\le\nu\le m).
\end{equation}
Indeed, for fixed $m$ the function
\[
 t\longmapsto
 \log\Gamma(t+1)+\log\Gamma\left(\frac m2-t+1\right)
 \qquad\left(0\le t\le\frac m2\right)
\]
is convex and symmetric.  Its maximum is therefore attained at an
endpoint, where the product is $\Gamma(m/2+1)$.

Now the case $m=0$ of the corollary is immediate.  Let $m\ge1$ and suppose that
\eqref{eq:r-bound} is known for all smaller indices.  From
\eqref{eq:r-log-recurrence}, \eqref{eq:gamma-product}, and
\eqref{eq:lambda-scalar-target},
\begin{align*}
 r_m
 &\le\frac1m\sum_{\nu=1}^m
 \nu\lambda_\nu
 6^{m-\nu}\Gamma\left(\frac{m-\nu}{2}+1\right)\\
 &\le\frac{6^m\Gamma(m/2+1)}m
 \sum_{\nu=1}^m
 \frac{\nu\lambda_\nu}
 {6^\nu\Gamma(\nu/2+1)}\\
 &<\frac1m6^m\Gamma\left(\frac m2+1\right)
 \le6^m\Gamma\left(\frac m2+1\right).
\end{align*}
This proves the required bound for $r_m$ in
\eqref{eq:auxiliary-summary}.
\end{proof}

\subsubsection{Proof of Lemma \ref{lem:central ineq}}

It remains to show \eqref{eq:lambda-scalar-target}. This will be done by a sequence of auxiliary results.
\begin{lem}\label{lem:lambda-majorant}
For every $m\ge2$,
\begin{equation}\label{eq:lambda-majorant}
 \lambda_m\le
 \frac{1405}{264}5^m
 +\sum_{b=2}^{\lfloor m/2\rfloor}
 \frac{8\cdot 5^{m-2b}(m-b-1)!}{\pi^b(m-2b)!}.
\end{equation}
\end{lem}

\begin{proof}
For a formal series $F$ in $u,z,w$, let
\[
 |F|_\#(u):=\sum_{m\ge0}\bigl\|[u^m]F\bigr\|_+u^m.
\]
We write $F\preccurlyeq G$ when $|F|_\#$ is coefficientwise bounded
by a series $G(u)$ with nonnegative coefficients.  Truncation modulo
$w^{15}$ can only improve such a bound.

Recall that $E_\ell:= \frac 1 \ell \sum_{d\mid \ell}\mu(\ell/d) u^{-d}$ and write
\[
 W_\ell(w):=\frac1\ell\sum_{d\mid\ell}\mu(\ell/d)(1-w^d)
 \qquad\mbox{and}\qquad
 \Delta_\ell:=\ell u^\ell E_\ell= \sum_{d\mid \ell}\mu(\ell/d) u^{\ell -d}.
\]
Every proper divisor of $\ell$ is at most $\ell/2$, and hence, for
$\ell\ge2$,
\[
 |\Delta_\ell-1|_\#
 \preccurlyeq\frac{u^{\lceil\ell/2\rceil}}{1-u}
 \preccurlyeq \frac{u}{1-u}.
\]
We also get 
\begin{equation}\label{equ:Delta inv}
|\Delta_\ell^{-1}|_\# 
=
|\sum_{k=0}^\infty (1-\Delta_\ell)^k |_\#
\preccurlyeq
\sum_{k=0}^\infty |1-\Delta_\ell|_\#^k
\preccurlyeq
\sum_{k=0}^\infty \left(\frac{u}{1-u}\right)^k
=
\frac{1-u}{1-2u}.
\end{equation}

Define 
\begin{equation}\label{eq:A-Y}
 A(u):=\frac{u^2(2-u)}{(1-u)(1-2u)}
 \qquad\mbox{and}\qquad
 Y(u):=\frac{u}{1-2u}+2A(u).
\end{equation}

Since $E_\ell^{-1}=\ell u^\ell\Delta_\ell^{-1}$, summation over
$\ell\ge2$ gives
\[
\sum_{\ell\ge2}|E_\ell^{-1}|_\#
\preccurlyeq
\frac{1-u}{1-2u}
\sum_{\ell\ge2} 
\ell u^\ell 
=
\frac{1-u}{1-2u}
\frac{u^2(2-u)}{(1-u)^2}
=
A(u).
\]
Since $\|W_\ell\|_+\le2$ we get
\begin{equation}
    \label{eq:E-inverse-majorantsA}
\sum_{\ell\ge2}|E_\ell^{-1}W_\ell|_\# \preccurlyeq2A(u).
\end{equation}
Next, $|\mu(\ell)/\ell|\le1/\ell$ and hence
\begin{equation}
    \label{eq:E-inverse-majorantsZ}
\sum_{\ell\ge2} |E_\ell^{-1} \frac{\mu(\ell)z}{u\ell}|_\#
\preccurlyeq 
\frac{1-u}{1-2u} \sum_{\ell \geq 2} u^{\ell-1} 
=
\frac{u}{1-2u} \preccurlyeq  Y(u),
\end{equation}
and 
\begin{equation}
    \label{eq:E-inverse-majorantsY}
\sum_{\ell\ge2}|E_\ell^{-1}(W_\ell+\frac{\mu(\ell)z}{\ell u})|_\#
 \preccurlyeq Y(u).
\end{equation}

Next, we apply these estimates to the terms of $\log \Rcal$. Concretely, from \eqref{eq:Rcal-def} and \eqref{eq:Uell-expanded} we have
\begin{equation}\label{equ:logR recap}
\begin{aligned}
\log\Rcal &= 
(1-w)\log(1-u) + \sum_{\ell\geq 2} W_\ell(\log(1-u^\ell)+\log \Delta_\ell) 
+\frac12 \sum_{\ell\geq 2} \sum_{k\geq 1} \frac{E_\ell^{-k}(X_\ell^k -Z_\ell^k) }{k}
\\&\quad
- \sum_{\ell\geq 2} \sum_{k\geq 2} \frac{E_\ell^{1-k}(X_\ell^k -Z_\ell^k)}{k(k-1)}
-\sum_{\ell\geq 2} \sum_{\substack{r\ge2\\r\text{ even}}}
 \frac{B_r}{r(r-1)}
 \sum_{k\ge1}\binom{r+k-2}{k}E_\ell^{1-r-k}(X_\ell^k-Z_\ell^k).
\end{aligned}
\end{equation}
with $X_\ell:=W_\ell+\frac{\mu(\ell)z}{\ell u}$, $Z_\ell:=\frac{\mu(\ell)z}{\ell u}$. We call the final triple sum the Bernoulli terms, and the remainder the non-Bernoulli terms. We begin by estimating the non-Bernoulli terms. First,
\[
| (1-w)\log(1-u)|_\# 
\preccurlyeq 2 \sum_{n\geq 1}\frac{u^n}{n}\preccurlyeq \frac{2u}{1-u}.
\]
Similarly, using again $\|W_\ell\|_+\le2$ we get 
\[
\sum_{\ell\geq 2} W_\ell \log(1-u^\ell)
\preccurlyeq
2\sum_{\ell\geq 2} \frac{u^\ell}{1-u^\ell}
\preccurlyeq 2\sum_{n\geq 2} (n-1)u^n
=
\frac{2u^2}{(1-u)^2}.
\]
Next,
\[
|\log \Delta_\ell |_\#  
\preccurlyeq \sum_{n\geq 1} \frac 1 n |1-\Delta_\ell|_\#^n
\preccurlyeq\sum_{n\geq 1} |1-\Delta_\ell|_\#^n
\preccurlyeq |1-\Delta_\ell|_\# \frac{1-u}{1-2u},
\]
using the same derivation as in \eqref{equ:Delta inv}. 
Furthermore, 
\[
\sum_{\ell \geq 2} |1-\Delta_\ell|_\#
\preccurlyeq\sum_{\ell \geq 2}
\sum_{\substack{d\mid \ell \\ d<\ell}} u^{\ell-d}
\preccurlyeq \sum_{n\geq 1}n u^n = \frac{u}{(1-u)^2}.
\]
It follows that
\[
|\sum_{\ell\geq 2} W_\ell \log \Delta_\ell|_\#
\preccurlyeq 2\frac{1-u}{1-2u} \sum_{\ell\geq 2} |1-\Delta_\ell|_\#
\preccurlyeq \frac{2u}{(1-u)(1-2u)}.
\]
Next consider the first double sum in \eqref{equ:logR recap}. We write 
\[
E_\ell^{-k} (X_\ell^k-Z_\ell^k) = E_\ell^{-k}(X_\ell-Z_\ell)\sum_{j=0}^{k-1}X_\ell^{k-1-j}Z_\ell^j
=
E_\ell^{-1} W_\ell\sum_{j=0}^{k-1}(E_\ell^{-1}X_\ell)^{k-1-j}(E_\ell^{-1}Z_\ell)^j.
\]
Using the estimates \eqref{eq:E-inverse-majorantsA}, \eqref{eq:E-inverse-majorantsZ} and \eqref{eq:E-inverse-majorantsY} we compute
\begin{align*}
&\frac 12\left|\sum_{\ell\geq 2} \sum_{k\geq 1} \frac{E_\ell^{-k}(X_\ell^k -Z_\ell^k) }{k} \right|_\#
\preccurlyeq
\sum_{k\geq 1} \frac1{2k} \sum_{\ell\geq 2} |E_\ell^{-1} W_\ell|_\# \sum_{j=0}^{k-1}|E_\ell^{-1}X_\ell|_\#^{k-1-j}|E_\ell^{-1}Z_\ell|_\#^j
\\& 
\preccurlyeq
\sum_{k\geq 1} \frac1{2k}\sum_{j=0}^{k-1}
\left(\sum_{\ell\geq 2} |E_\ell^{-1} W_\ell|_\#\right)
\left(\sum_{\ell\geq 2} |E_\ell^{-1} X_\ell|_\#\right)^{k-1-j}
\left(\sum_{\ell\geq 2} |E_\ell^{-1} Z_\ell|_\#\right)^j
\\&
\preccurlyeq
\sum_{k\geq 1}\frac1{2k}\sum_{j=0}^{k-1}
2A Y^{k-1-j} Y^j =
\sum_{k\geq 1}\frac1{2k} k\cdot 
2A Y^{k-1} =
\frac{A}{1-Y}.
\end{align*}

Now consider the second double sum in \eqref{equ:logR recap}. As before, we have 
\[
E_\ell^{1-k}(X_\ell^k -Z_\ell^k)
=W_\ell \sum_{j=0}^{k-1} (E_\ell^{-1}X_\ell)^{k-1-j}(E_\ell^{-1}Z_\ell)^j,
\]
and proceeding similarly, we get
\begin{align*}
&\left|\sum_{\ell\geq 2} \sum_{k\geq 2} \frac{E_\ell^{1-k}(X_\ell^k -Z_\ell^k)}{k(k-1)} \right|_\#
\preccurlyeq 
\sum_{k\geq 2} \frac{1}{k(k-1)} 2k\cdot Y^{k-1}
\preccurlyeq \frac{2Y}{1-Y}.
\end{align*}

So all non-Bernoulli terms are together bounded by
\begin{equation}\label{eq:non-Bernoulli-majorant}
 \begin{split}
 N(u):={}&
 \frac{2u}{1-u}
 +\frac{2u^2}{(1-u)^2}
 +\frac{2u}{(1-u)(1-2u)}\\
 &+\frac{A(u)}{1-Y(u)}
 +\frac{2Y(u)}{1-Y(u)}.
 \end{split}
\end{equation}
All coefficients of $N(u)$ are nonnegative, and hence in particular $N(1/5)\geq 5^{-m} [u^m]N(u)$.  Direct
substitution gives
\[
 Y(1/5)=\frac{19}{30},
 \qquad
 N(1/5)=
 \frac12+\frac18+\frac56+\frac9{22}+\frac{38}{11}
 =\frac{1405}{264}.
\]
It follows from the definition \eqref{eq:non-Bernoulli-majorant} of $N$ that
\[
 [u^m]N(u)\le \frac{1405}{264}5^m,
\]
which gives the first term on the right-hand side of
\eqref{eq:lambda-majorant}.

It remains to bound the Bernoulli part of
\eqref{equ:logR recap}, i.e., the final triple sum. 
We may proceed as above:
\begin{align*}
&\left|\sum_{\ell\geq 2} \sum_{r\ge2}
 \frac{B_r}{r(r-1)}
 \sum_{k\ge1}\binom{r+k-2}{k}E_\ell^{1-r-k}(X_\ell^k-Z_\ell^k)\right|_\#
 \\&\preccurlyeq
 \sum_{r\ge2}
 \frac{|B_r|}{r(r-1)}
 \sum_{k\ge1}\binom{r+k-2}{k}
 \sum_{\ell\geq 2}
 \sum_{j=0}^{k-1}
 |E_\ell^{-r}|_\#|W_\ell|_\# 
 |E_\ell^{-1}X_\ell|^{k-1-j}_\#
 |E_\ell^{-1}Z_\ell|^{j}_\#
\\& 
\preccurlyeq
\sum_{r\ge2}
 \frac{|B_r|}{r(r-1)}
 \sum_{k\ge1}\binom{r+k-2}{k}
 2k A^r Y^{k-1}.
\end{align*}
By differentiating the binomial expansion of $(1-y)^{1-r}$ we obtain the identity
\[
 \sum_{k\ge1}k\binom{r+k-2}{k}y^{k-1}
 =(r-1)(1-y)^{-r}.
\]
Hence the contribution of Bernoulli index $r$ is bounded by
\[
 \frac{2|B_r|}{r}
 \left(\frac{A(u)}{1-Y(u)}\right)^r.
\]
The rational series in parentheses satisfies
\begin{equation}\label{eq:Bernoulli-rational-majorant}
 \frac{A(u)}{1-Y(u)}
 =\frac{u^2(2-u)}{1-4u-u^2+2u^3}
 \preccurlyeq\frac{2u^2}{1-5u}.
\end{equation}
To verify the last inequality explicitly, write the series on the
left as $\sum_{n\ge2}q_nu^n$.  Its coefficients are nonnegative,
$q_2=2$, $q_3=7$, and its denominator gives
\[
 q_n=4q_{n-1}+q_{n-2}-2q_{n-3}\qquad(n\ge4).
\]
Dropping the last, nonpositive term and applying induction gives
$q_n\le2\cdot5^{n-2}$, which is the claimed coefficientwise
domination.  For even $r$, the standard estimate \cite[Equation 24.9.8]{NIST:DLMF}
\[
 |B_r|\le\frac{4r!}{(2\pi)^r}
\]
and \eqref{eq:Bernoulli-rational-majorant} give, for $m\ge2r$,
\begin{align*}
 [u^m]\frac{2|B_r|}{r}
 \left(\frac{A(u)}{1-Y(u)}\right)^r
 &\le
 \frac{8(r-1)!}{\pi^r}
 5^{m-2r}\binom{m-r-1}{r-1}\\
 &=\frac{8\,5^{m-2r}(m-r-1)!}
 {\pi^r(m-2r)!}.
\end{align*}
The Bernoulli contributions with odd $r>1$ vanish.  Enlarging the
majorant by inserting the same positive bound for odd $r$ gives the
sum in \eqref{eq:lambda-majorant}. We also replaced the index $r$ by $b$, to avoid confusion with the quantities $r_m$ of \eqref{eq:auxiliary-summary}.
\end{proof}

For brevity, denote the right-hand side of
\eqref{eq:lambda-majorant} by $\lambda_m^{\mathrm{maj}}$, and put
\begin{equation}\label{eq:T-def}
 T_m:=
 \frac{m\lambda_m^{\mathrm{maj}}}
 {6^m\Gamma(m/2+1)}
 =\alpha_m+\sum_{b=2}^{\lfloor m/2\rfloor}\beta_{m,b},
\end{equation}
where
\begin{equation}\label{eq:alpha-beta-def}
 \alpha_m:=\frac{1405}{264}
 \frac{m(5/6)^m}{\Gamma(m/2+1)}
 \qquad\mbox{and}\qquad
 \beta_{m,b}:=
 \frac{8m\,5^{m-2b}(m-b-1)!}
 {\pi^b6^m(m-2b)!\Gamma(m/2+1)}.
\end{equation}
The following computational verification records all finite calculations needed in this
subsection.

\begin{inputbox}
\begin{computerverification}[\texttt{lambda\_t\_bounds.py}]
\label{input:R-bound}
The checker reconstructs from \eqref{eq:Lambda-def} the exact values \eqref{eq:lambda-def}
\[
\begin{array}{c|cccccccccc}
 m&1&2&3&4&5&6&7&8&9&10\\ \hline
 \lambda_m&3&\frac{23}6&\frac{29}6&\frac{20}3&\frac{141}{10}&
 \frac{403}{18}&\frac{1985}{42}&\frac{10703}{120}&
 \frac{5273}{30}&\frac{256937}{660}
\end{array}
\]
and verifies the three scalar inequalities
\begin{equation}\label{eq:finite-lambda-checks}
 \sum_{m=1}^{10}
 \frac{m\lambda_m}{6^m\Gamma(m/2+1)}<\frac{17}{20},
 \qquad
 T_{11}<\frac1{32},
 \qquad
 T_{12}<\frac1{96}.
\end{equation}
The quantities $T_{11}$ and $T_{12}$ are evaluated from the
explicit formulas \eqref{eq:T-def} and \eqref{eq:alpha-beta-def}; no
other value of $T_m$ is evaluated.  The bounds $\pi>3$ and
$\sqrt\pi>7/4$ already reduce all three displayed comparisons to
finite rational inequalities.
\end{computerverification}
\end{inputbox}

All remaining coefficients are controlled by the following elementary
contraction.

\begin{lem}\label{lem:T-contraction}
For every $m\ge11$,
\[
 T_{m+2}<\frac17T_m.
\]
\end{lem}

\begin{proof}
The first terms in \eqref{eq:T-def} satisfy
\[
 \frac{\alpha_{m+2}}{\alpha_m}
 =\frac{25}{18m}\le\frac{25}{198}.
\]
For $2\le b\le\lfloor m/2\rfloor$, pair the $b$th summand at $m$
with the $(b+1)$st summand at $m+2$.  Their ratio is
\[
 \frac{\beta_{m+2,b+1}}{\beta_{m,b}}
 =\frac{m-b}{18\pi m}<\frac1{54}.
\]
These pairs account for every term $\beta_{m+2,b}$ except the new term
with $b=2$.  Relative to $\alpha_m$, this new term satisfies
\[
 \frac{\beta_{m+2,2}}{\alpha_m}
 =\frac{352(m-1)}{105375\pi^2m}<\frac1{2500},
\]
where the last inequality follows from $\pi>3$.  Consequently,
using
\[
 \frac{25}{198}+\frac1{2500}<\frac17,
 \qquad
 \frac1{54}<\frac17,
\]
we obtain
\begin{align*}
 T_{m+2}
 &=\alpha_{m+2}+\beta_{m+2,2}
   +\sum_{b=2}^{\lfloor m/2\rfloor}\beta_{m+2,b+1}\\
 &<\left(\frac{25}{198}+\frac1{2500}\right)\alpha_m
   +\frac1{54}\sum_{b=2}^{\lfloor m/2\rfloor}\beta_{m,b}
 <\frac17T_m. \qedhere
\end{align*}
\end{proof}

\begin{proof}[Proof of Lemma \ref{lem:central ineq}]
Combining Lemmas~\ref{lem:lambda-majorant} and~\ref{lem:T-contraction}
with Computer Verification~\ref{input:R-bound} gives
\begin{align*}
 \sum_{m=1}^{\infty}
 \frac{m\lambda_m}{6^m\Gamma(m/2+1)}
 &\le
 \sum_{m=1}^{10}
 \frac{m\lambda_m}{6^m\Gamma(m/2+1)}
 +\sum_{m=11}^{\infty}T_m\\
 &<\frac{17}{20}
 +\frac{T_{11}+T_{12}}{1-1/7}\\
 &<\frac{17}{20}
 +\frac76\left(\frac1{32}+\frac1{96}\right)
 =\frac{647}{720}<\frac9{10}<1.
\end{align*}
In particular, the finite verification in
\eqref{eq:finite-lambda-checks} and the displayed contraction account
for every term in the infinite sum.
\end{proof}

\bibliographystyle{amsalpha}
\bibliography{refs}
\par\vspace{\baselineskip}
\end{document}